\documentclass[11pt, english, a4paper]{article}
\usepackage{fullpage}
\usepackage{indentfirst}
\usepackage{hyperref}
\hypersetup{colorlinks=true,citecolor=blue,linkcolor=blue,urlcolor=blue}
\usepackage{amssymb,amsmath,amsfonts,amsthm,latexsym,enumerate,url,cases}
\usepackage{stmaryrd}
\allowdisplaybreaks[4]
\usepackage{mathrsfs}
\numberwithin{equation}{section}
\usepackage{stmaryrd}
\SetSymbolFont{stmry}{bold}{U}{stmry}{m}{n}
\usepackage{graphicx}
\usepackage{cite}
\usepackage{appendix}
\usepackage{color}
\numberwithin{equation}{section}

\newtheorem{theorem}{Theorem}[section] %
\newtheorem{lemma}{Lemma}[section] %
\newtheorem{definiton}{Definition}[section] %
\newtheorem{remark}{Remark}[section] %

\addappheadtotoc
\usepackage{cite}
\title{Wong-Zakai type approximations of rough random dynamical systems by smooth noise
\thanks
{This work is supported in part by a NSFC Grant No. 12171084 and the fundamental Research Funds for the Central Universities No. 2242022R10013.}}
\author{Qiyong Cao$^1$\thanks{E-mail: xjlyysx@163.com}, Hongjun Gao$^2$\thanks{E-mail: gaohj@hotmail.com, corresponding author.} and Bj\"{o}rn Schmalfuss$^3$\thanks{E-mail:  bjoern.schmalfuss@uni-jena.de}\\ {\small \it 1. School of Mathematical Sciences, Nanjing Normal University, Nanjing 210023, P. R. China}\\
 {\small \it $2$. School of Mathematics, Southeast University, Nanjing 211189, P. R. China}
\\ {\small \it $3.$ Institute of Stochastics,
Friedrich-Schiller-University, D-07743 Jena, Germany}}
\date{}

\begin{document}

\maketitle
\abstract{This paper is devoted to the smooth and stationary  Wong-Zakai approximations for a class of rough differential equations driven by a geometric fractional Brownian rough path $\boldsymbol{\omega}$ with Hurst index $H\in(\frac{1}{3},\frac{1}{2}]$. We first construct the approximation $\boldsymbol{\omega}_{\delta}$ of $\boldsymbol{\omega}$ by probabilistic arguments, and then using the rough path theory to obtain the Wong-Zakai approximation for the solution on any finite interval. Finally, both the original system and approximative system  generate  a continuous  random dynamical systems $\varphi$ and $
\varphi^{\delta}$. As a consequence of the Wong-Zakai approximation of the solution,  $\varphi^{\delta}$ converges to $\varphi$ as $\delta\rightarrow 0$.}


\section{Introduction}\label{section-introduction}
In this paper, we  consider the following rough differential equation
\begin{equation}
dy=(Ay+f(y))dt+g(y)d \boldsymbol{\omega},
\end{equation}
where $\boldsymbol{\omega}$ is a geometric fractional Brownian rough  path, 
and $f,A,g$ are defined below.

\smallskip

For the study of the dynamics of stochastic differential equations,
there are many well known methods. One of methods  is based on an approximation argument.
In this paper, we consider stochastic differential equations with a smooth driver to approximate the original equation. Our idea is to follow
Wong and Zakai \cite{MR195142,MR0183023}.  Wong and Zakai  studied  the
piecewise linear approximations and the piecewise smooth approximations for
a one dimensional Brownian motion.  Their work was extended to higher dimension Brownian motion \cite{MR0402921,MR0400425,MR0458587,MR1011252,MR3456344}. However,  the solutions of  approximated equations do not generate continuous random dynamical systems.

\smallskip
We should keep a particular approximation scheme\cite{MR2831717,MR2727178} for the driving process in mind.
Let $(\Omega,\mathcal{F},\mathbb{P})$ be a
probability space and $\theta$ be a Wiener shift over
$(\Omega,\mathcal{F},\mathbb{P})$.
$\mathcal{G}_{\delta}(\theta_{t}\omega)$ is an  approximative process of
white noise which has the following form:
\begin{equation*}
\mathcal{G}_{\delta}(\theta_{t}\omega)=\frac{1}{\delta}\theta_{t}\omega(\delta)=\frac{1}{\delta}(\omega_{t+\delta}-\omega_{t}).
\end{equation*}
In addition, let
\begin{equation*}
W_{\delta}(t,\omega)=\frac{1}{\delta}\int_{0}^{t}\theta_{s}\omega_\delta ds,
\end{equation*}
it is a smooth Gaussian process with stationary increments and approximates Brownian motion.

\smallskip
Brownian motion has many nice
properties, such as independent increments, Markov property, martingale property. But the
fractional Brownian motion does not have these
properties. It is neither a Markov process or a martingale and
increments are not independent. Thus the Wong-Zakai approximation of the
fractional Brownian motion is  worth being studied. There are
interesting results for the  Wong-Zakai approximations of the fractional
Brownian motion \cite{MR2486926,MR1780221,MR2471936,MR3161527,MR3501366}. 

\smallskip

Based on the articles we mentioned here, it is a natural question to
consider the Wong-Zakai approximation of the geometric fractional Brownian rough
path $\boldsymbol{\omega}=(\omega^{1},\omega^{2})$. Thus our plan is to
consider
$\boldsymbol{\omega}_{\delta}=(W_{\delta}(\cdot,\omega^{1}),\mathbb{W}_{\delta}(\omega^{1}))$,
where
\begin{equation*}
	\begin{aligned}
W^{i}_{\delta}(t,\omega^{1})&=\frac{1}{\delta}\int_{0}^{t}\theta_{r}\omega^{1,i}dr,\\
	\mathbb{W}^{i,j}_{\delta}(\omega^{1})_{s,t}&=\int_{s}^{t}W^{i}_{\delta}(\cdot,\omega^{1})_{s,r}dW^{j}_{\delta}(r,\omega^{1})
	\end{aligned}
\end{equation*}
  for $1\leq i,j\leq d$, and $s<t\in[-T,T]$. Gao et al.\cite{MR4266219}
established the  Wong-Zakai approximation of the Brownian rough path, namely
$H=\frac{1}{2}$. One of our main purposes  is to extend these results to the
fractional Brownian rough path with Hurst index
$\frac{1}{3}<H<\frac{1}{2}$. The main difficulty  is that we do not have
the concavity of functions
$\sigma_{W^{i}_{\delta}(\cdot,\omega^{1})}(u)=\mathbb{E}(W^{i}_{\delta}(t+u,\omega^{1})-W^{i}_{\delta}(t,\omega^{1}))^{2}$
and
$\sigma_{X_{\delta}^{i}}(u)=\mathbb{E}(X_{\delta}^{i}(t+u)-X_{\delta}^{i}(t))^{2}$,
$X_{\delta}^{i}(t)=\omega^{1}(t)-W_{\delta}(t,\omega^{1})$.  In order to
overcome this difficulty, we make use of the properties of the fractional
Brownian motion to construct the convergence between
$\boldsymbol{\omega}$ and $\boldsymbol{\omega}_{\delta}$ in the sense of
almost surely. Furthermore, compared with \cite{MR4266219},  our result is more general. That is to say, we get $\boldsymbol{\omega}_{\delta}\rightarrow \boldsymbol{\omega}$ as $\delta\rightarrow 0$ rather than $\boldsymbol{\omega}_{\delta_{i}}\rightarrow \boldsymbol{\omega}$ as $i\rightarrow \infty$, where  the sequence $\{\delta_{i}\}_{i\in\mathbb{N}}$  converges  sufficiently fast to 0 as  $i\rightarrow \infty$,  and the convergence   also  holds  for Hurst index $H\in(\frac{1}{3},\frac{1}{2})$, not just $H=\frac{1}{2}$.

\smallskip
It is very important to study such approximation. Firstly, compared with the piecewise approximations, an advantage of such approximation is that the approximated equations \eqref{5.2} generate a random dynamical system and the solutions of approximated equations \eqref{5.2} converge to  the solution of  rough differential equation \eqref{5.1}. Secondly, the approximation $W_{\delta}(t,\omega)$  has been used for stochastic ordinary  and partial differential equations when noise  is very simple(additive noise or linear multiplicative noise), and the dynamical behaviour of  approximated systems  converge to  the behaviour of the original system, such as,  invariant manifolds\cite{MR3912728,MR3680943}.  So  we can use  $\boldsymbol{\omega}_\delta$ as the approximation of dynamics for rough case. Thirdly, based on the techniques used in this paper, the conjugate transformation method \cite{MR1723992} is not necessary for the convergence of the dynamical behaviour, and more complicated noises can be considered.

\smallskip

Our another main object  is to construct the Wong-Zakai approximation
for rough differential equation \eqref{5.1}. To this end, we  first construct
the Wong-Zakai approximation for a fractional Brownian rough path $\boldsymbol{\omega}$ on any finite interval.
As the byproduct of the Wong-Zakai approximation of the solution, the random dynamical system $\varphi^{\delta}$ generated by \eqref{5.2}
converge to $\varphi$  as $\delta\rightarrow 0$, where $\varphi$ is generated by \eqref{5.1}.  
Finally, for rough differential equations which contain a drift term,  there is no results that the stability of the solution map with respect to the noise, so we dedicate a significant portion of the paper to discuss the Wong-Zakai approximation of the solution.  Friz and Hairer\cite[Theorem 8.15]{MR4174393} shows that flows generated by the solutions of rough differential equations without the drift terms are stable under some  conditions.  For rough differential equations with the drift terms,   Friz and Victoir\cite[Theorem 12.10]{MR2604669} required that the drift term is at least differentiable and its derivative is bounded, but the drift term
 is  only Lipschitz in our paper.  Riedel and Scheutzow\cite[Theorem 3.1, 4.3]{MR3567487} assumed  the drift term is locally Lipschitz  and linear growth,  and the  diffusion term $g(y)$ w.r.t. $y$ is $C^{\gamma}(R^m),\gamma>3$ and its all derivatives $D^ig(y),i=0,1,2,3$ and  $(\gamma-3)$-H\"{o}lder semi-norm  of $D^3g(y)$ are uniform bounded. However, the diffusion term is three times differentiable and all  derivatives are uniform bounded  in our paper.   In addition,  these two theorems are proved in  the sense of Friz-Victoir\cite{MR2604669}, they did not need rough integral to be specified(see \cite[Definition 10.17]{MR2604669}).   Based on those reasons, we give a detailed proof for the Wong-Zakai approximation of the solutions in our paper.

\smallskip


Our paper is organized as follows. In Section 2,  we present the main tools of
rough path theory and random dynamical system.  In Section 3, we analyze  the approximation of the geometric fractional
Brownian motion rough path. In particular, these approximations are smooth
Gauss-processes with stationary increments. 
Section 4 is devoted to the
Wong-Zakai approximation  of  the solution of  the  original rough equation. 
The convergence  also holds for the associated random dynamical systems. In an Appendix, we give
some necessary theorems to construct the Wong-Zakai approximation for noises and solutions. Furthermore, the facts in Appendix A are taken from
\cite{MR4174393}.

\section{Preliminaries}
In this section, we will recall some facts about rough paths and random dynamical systems.  The constant $C$  may change form line to line.  For a compact time interval $I=[T_{1},T_{2}]\subset\mathbb{R}$, we write $|I|=T_{2}-T_{1}$ and $I^{2}=\left\{(s,t)\in I\times I:s\leq t\right\}$. We denote by $C(I;\mathbb{R}^{m})$ the space of all continuous paths $y:I\rightarrow \mathbb{R}^{m}$ equipped with the norm $\|\cdot\|_{\infty,I}$ given by $\|y\|_{\infty,I}=\sup_{t\in I}\|y_{t}\|$, where $\|\cdot\|$ is the  Euclidean norm, and let  $C^1(I,R^{d})$ denote the space of the first order differentiable $R^d$-valued functions on $I$. We write $y_{s,t}=y_{t}-y_{s}$. For $p\geq1$, denote by $C^{p-var}(I;\mathbb{R}^{d})$  all continuous paths $y: I\rightarrow \mathbb{R}^{d}$ which have a finite $p$-variation
\begin{equation}
\begin{aligned}
\interleave y\interleave_{p-var,I}=\left(\sup_{\mathcal{P}(I)}\sum_{[t_{i},t_{i+1}]\in\mathcal{P}(I)}\|y_{t_{i},t_{i+1}}\|^{p}\right)^{\frac{1}{p}}<\infty,
\end{aligned}
\end{equation}
where $\mathcal{P}(I)$ is a  partition of the interval $I$. Furthermore, we equip this space with the norm
\begin{equation}
\begin{aligned}
\|y\|_{p-var,I}:=\|y_{T_{1}}\|+\interleave y\interleave_{p-var,I}.
\end{aligned}
\end{equation}
This norm is equivalent to
 \begin{equation}
\begin{aligned}
\|y\|_{p-var,I}:=\|y\|_{\infty,I}+\interleave y\interleave_{p-var,I}.
\end{aligned}
\end{equation}
For properties of the $p-$variarion norm we refer to \cite{MR2604669}.
\begin{lemma}\label{Lemma 2.1}
Let $\omega\in C^{p-var}([s,t];R^{d})$, $p>1$. For any  partition $\mathcal{P}(s,t)$ of the interval $[s,t]$ given by $s=u_{1}<u_{2}<\cdots<u_{n}=t$, we have
\begin{equation}
\sum_{i=1}^{n-1}\interleave \omega\interleave_{p-var,[u_{i},u_{i+1}]}^{p}\leq \interleave \omega\interleave_{p-var,[s,t]}^{p}\leq(n-1)^{p-1}\sum_{i=1}^{n-1}\interleave \omega\interleave_{p-var,[u_{i},u_{i+1}]}^{p}.
\end{equation}
\end{lemma}
Moreover, for any $\alpha\in (0,1)$, we denote by $C^{\alpha}(I,\mathbb{R}^{m})$ the space of H\"{o}lder continuous functions with H{\"o}lder exponent $\alpha$, and equipped with the norm
\begin{equation}
\begin{aligned}
\|y\|_{\alpha,I}:=\|y_{T_{1}}\|+\interleave y\interleave_{\alpha,I}
\end{aligned}
\end{equation}
or the equivalent norm
\begin{equation}
\begin{aligned}
\|y\|_{\alpha,I}:=\|y\|_{\infty,I}+\interleave y\interleave_{\alpha,I},
\end{aligned}
\end{equation}
where  $\interleave y\interleave_{\alpha,I}=\sup_{s<t\in I}\frac{\|y_{s,t}\|}{(t-s)^{\alpha}}<\infty$.
\begin{definiton}
For $\alpha\in(\frac{1}{3},\frac{1}{2}]$, A pair $\boldsymbol{X}=(X,\mathbb{X})\in C^{\alpha}(I,\mathbb{R}^{d})\oplus C^{2\alpha}(I^{2},\mathbb{R}^{d\times d})$   is called rough path if it satisfies the analytic relations
\begin{equation}\label{2.8}
\interleave X\interleave_{\alpha}:=\sup_{s<t\in I}\frac{|X_{s,t}|}{|t-s|^{\alpha}}<\infty,\quad \interleave \mathbb{X}\interleave_{2\alpha}:=\sup_{s<t\in I}\frac{|\mathbb{X}_{s,t}|}{|t-s|^{2\alpha}} <\infty
\end{equation}
and Chen's relation
\begin{equation}
\begin{aligned}
\mathbb{X}_{s,t}=\mathbb{X}_{s,u}+\mathbb{X}_{u,t}+X_{s,u}\otimes X_{u,t}
\end{aligned}
\end{equation}
for $s\leq u\leq t\in I$, we denote by $\mathcal{C}^{\alpha}(I,R^{d})$ the space of rough path. In addition, for any $X\in C^1(I,R^d)$, there is a canonical lift $S(X):=(X,\mathbb{X})$ in $\mathcal{C}^{\alpha}(I,R^{d})$  defined as
\begin{equation*}
\mathbb{X}_{s,t}^{k,l}=\int_{s}^{t}\int_{s}^{r}dX^k_{r^\prime}dX^l_r,\quad s<t\in I \text{and} ~ k,l\in\{1,\cdots,d\}.
\end{equation*}
We denote $\mathcal{C}_{g}^{\alpha}(I,R^{d})$ by the geometric rough space, i.e. the closure of  the  canonical lift $S(X),X\in C^{1}(I,R^d)$.
\end{definiton}

The first component $X$ is the path component and the second component $\mathbb{X}$ is called L\'{e}vy area or the
second order process. 
Let  $\mathcal{C}^{\alpha}(I,\mathbb{R}^{d})$ be equipped with the homogeneous norm\cite[page 18]{MR4174393}
\begin{equation}
\begin{aligned}
\interleave \boldsymbol{X}\interleave_{\alpha,I}=\interleave X \interleave_{\alpha,I}+\interleave \mathbb{X}\interleave^{\frac{1}{2}}_{2\alpha,I^{2}}.
\end{aligned}
\end{equation}
In addition, we can use  the $p$-variation norm\cite[page 165]{MR2604669}
\begin{equation}
\begin{aligned}
\interleave \boldsymbol{X}\interleave_{p-var,I}=\left(\interleave X \interleave^{p}_{p-var,I}+\interleave \mathbb{X}\interleave^{q}_{q-var,I^{2}}\right)^{\frac{1}{p}}
\end{aligned}
\end{equation}
to describe a rough path $\boldsymbol{X}$, where $\interleave \mathbb{X}\interleave_{q-var,I^{2}}=\left(\sup_{\mathcal{P}(I)}\sum_{[t_{i},t_{i+1}]\in \mathcal{P}(I)}\|\mathbb{X}_{t_{i},t_{i+1}}\|^{q}\right)^{\frac{1}{q}}$ and $q=\frac{p}{2},p\geq 1$. Let  $\mathcal{C}^{p-var}(I,R^{d})$ denote   the space of all rough paths which have a finite $p$-variation norm. It is clear that $\mathcal{C}^{\alpha}(I,\mathbb{R}^{d})\subset \mathcal{C}^{p-var}(I,R^{d})$ for $p=\frac{1}{\alpha}$.

\begin{definiton}
Let $(\Omega,\mathcal{F},\mathbb{P})$ be a probability space, we say the quadruple $(\Omega,\mathcal{F},\mathbb{P},(\theta_{t})_{t\in\mathbb{R}})$ is a metric dynamical system over $(\Omega,\mathcal{F},\mathbb{P})$, if the map $\theta:\mathbb{R}\times\Omega\rightarrow \Omega$ satisfies the following conditions
\begin{description}
  \item[1] the map $(t,\omega)\mapsto \theta_{t}\omega$ is measurable for $(\mathcal{B}(\mathbb{R})\otimes\mathcal{F},\mathcal{F})$;
  \item[2] $\theta_{0}=Id,\theta_{t}\circ\theta_{s}=\theta_{t+s}, t,s\in\mathbb{R}$;
  \item[3] $\mathbb{P}(\theta_{t}^{-1}B)=\mathbb{P}(B), B\in\mathcal{F},t\in\mathbb{R}$.
\end{description}
Furthermore, if for any $\theta$-invariant set $B\in\mathcal{F}$, namely, $\theta_{t}^{-1}B=B,t\in \mathbb{R},B\in\mathcal{F}$, we have that $P(B)=0$ or $P(B)=1$, then the metric dynamical system $(\Omega,\mathcal{F},\mathbb{P},(\theta_{t})_{t\in\mathbb{R}})$ is called ergodic and the measure  $\mathbb{P}$ is an ergodic measure.
\end{definiton}

\begin{remark}\label{remark 2.1}
It is well known that the canonical fractional Brownian motion $B^{H}(t,\omega):=\omega_{t}$ in $R^{d}$  with Hurst index $H\in(0,1)$  is a centered, continuous Gaussian process with stationary increments and  covariance
$$R_{B^{H}}(s,t)=\frac{1}{2}(|t|^{2H}+|s|^{2H}-|t-s|^{2H})Id,\quad t,s\in R,$$
where $Id$ is the identity matrix. Firstly, we  consider the quadruple $(C_{0}(R,R^{d}),\\\mathcal{B}(C_{0}(R,R^{d})),\mathbb{P}_{H},\theta)$, where $C_{0}(R,R^{d})$ is the space of continuous functions which are zero at zero, equipped with the compact open topology.  $\mathcal{B}(C_{0}(R,R^{d}))$  is the Borel $\sigma$-algebra of $C_{0}(R,R^{d})$. $\mathbb{P}_{H}$ is the Gaussian distribution of the fractional Brownian motion $B^{H}$. $\theta_{t}$ is the so-called Wiener shift $\theta_{t}\omega(\cdot)=\omega(\cdot+t)-\omega(t)$. Then $(C_{0}(R,R^{d}),\mathcal{B}(C_{0}(R,R^{d})),\mathbb{P}_{H},\theta)$ is an ergodic metric dynamical system, see   \cite{MR2836654,MR1893308,MR3226746}. Secondary, a fractional Brownian motion has a H\"{o}lder continuous version. Then we have $\theta$-invariant set $\Omega$ of full measure  such that $\omega\in\Omega$ is
$\beta$-H\"{o}lder continuity $(\beta<H)$ on any interval $[-T,T]$. Let $\mathcal{F}=\Omega\bigcap \mathcal{B}(C_{0}(R,R^{d}))$, we still use the symbol $\mathbb{P}_{H}$ which is the restriction of $\mathbb{P}_{H}$ on $\mathcal{F}$. Then $(\Omega,\mathcal{F},\mathbb{P}_{H},\theta)$ is also an ergodic metric dynamical system.
\end{remark}
\begin{definiton}
Let $(\Omega,\mathcal{F},\mathbb{P},\theta)$  be a metric dynamical system. We call $\varphi:\mathbb{R}^{+}\times\Omega\times\mathbb{R}^{m}\rightarrow \mathbb{R}^{m}$  a random dynamical system on $(\Omega,\mathcal{F},\mathbb{P},\theta)$, if the following conditions hold
\begin{itemize}
  \item the mapping  $(t,\omega,x)\mapsto \varphi(t,\omega,x)$ is $(\mathcal{B}(R^{+})\otimes \mathcal{F}\otimes \mathcal{B}(R^{m}),\mathcal{B}(R^{m}))$-measurable;
  \item $\varphi(0,\omega,\cdot)=Id$ for any $\omega\in\Omega$;
  \item $\varphi(t_{1}+t_{2},\omega,\cdot)=\varphi(t_{2},\theta_{t_{1}}\omega,\varphi(t_{1},\omega,\cdot))$ for $t_{1},t_{2}\in \mathbb{R}^{+},\omega\in\Omega$.
\end{itemize}
\end{definiton}

\section{Approximation of the fractional Brownian motion by a smooth
process with stationary increments}\label{section-intersection}\label{approxiamtion of  the noise}
In this section, we will introduce an approximation scheme for the geometric
fractional Brownian rough path. Furthermore, the convergence between the
smooth rough path generated by the approximate process of fractional
Brownian motion and the rough path generated by the canonical lift of
the fractional Brownian motion is considered.

Let $X_{t}$ be a continuous, centered Gaussian process with values in $R^{d}$. For process $X_{t}$, its covariance can be represented as follows
\begin{equation*}
R_{X}(s,t)=EX_{s}\otimes X_{t}.
\end{equation*}
The rectangular increments of the covariance $R_{X}$  for   $u\leq v$ and   $ u^{\prime}\leq v^{\prime}$ are defined by
\begin{equation*}
R_{X}\left(\begin{array}{cc}
	u & v \\
	u^{\prime} & v^{\prime}
\end{array}\right):=\mathbb{E}(X(v)-X(u))\otimes\left(X\left(v^{\prime}\right)-X\left(u^{\prime}\right)\right).
\end{equation*}
We define the $\rho$-variation of  $R_{X}$ for $2>\rho\geq 1$ on the interval $[s,t]^{2}$ as follows:
\begin{equation*}
\|R_{X}\|_{\rho-var;[s,t]^{2}}:=\left(\sup_{\mathcal{P}(s,t)}\sup_{\mathcal{P}^{\prime}(s,t)}\sum_{[u,v]\in\mathcal{P}}\sum_{[u^{\prime},v^{\prime}]\in\mathcal{P}}\left|R_{X}\left(\begin{array}{cc}
	u & v \\
	u^{\prime} & v^{\prime}
\end{array}\right)\right|^{\rho}\right)^{\frac{1}{\rho}}.
\end{equation*}
In order to calculate  the distance between $\boldsymbol{\omega}$ and
$\boldsymbol{\omega}_{\delta}$, we need the following inhomogeneous
rough path metric:
\begin{equation*}
\rho_{\alpha,[-T,T]}(\mathbf{X},\mathbf{Y}):=\sup_{s\neq
	t\in[-T,T]}\frac{|X_{s,t}-Y_{s,t}|}{|t-s|^{\alpha}}+\sup_{s\neq
	t\in[-T,T]}\frac{|\mathbb{X}_{s,t}-\mathbb{Y}_{s,t}|}{|t-s|^{2\alpha}},
\end{equation*}
for $\mathbf{X},\mathbf{Y}\in\mathcal{C}^{\alpha}([-T,T];R^{d})\subset
\mathcal{C}^{p-var}([-T,T];R^{d})$, $p=\frac{1}{\alpha},\alpha\in\left(\frac{1}{3},\frac{1}{2}\right)$. In Section 4, we shall use the
$p$-variation metric, namely, for $p=\frac{1}{\alpha}$ and $q=\frac{p}{2}$,
\begin{equation*}
	\begin{aligned}
\rho_{p-var,[-T,T]}(\mathbf{X},\mathbf{Y}):&=\left(\sup_{\mathcal{P}([-T,T])}\sum_{[u,v]\in\mathcal{P}([-T,T])}|X_{s,t}-Y_{s,t}|^{p}\right)^{\frac{1}{p}}\\
	&+\left(\sup_{\mathcal{P}([-T,T]) }\sum_{[u,v]\in\mathcal{P}([-T,T])}|\mathbb{X}_{s,t}-\mathbb{Y}_{s,t}|^{q}\right)^{\frac{1}{q}}.
	\end{aligned}
\end{equation*}
Theorem \ref{theorem 3.1} shows that we can find a rough path $\boldsymbol{\omega}=(\omega^1,\omega^2)\in \mathcal{C}^{\beta}_g([0,T];R^d)$ for any $T>0$	, it  can further extend to  $\boldsymbol{\omega}\in \mathcal{C}^{\beta}_g([-T,T];R^d)$. Moreover,  Theorem \ref{theorem 3.1} also ensures that we can define a  version
\begin{equation*}
\theta_\tau\boldsymbol{\omega}:=(\theta_\tau\omega^1,\theta_{\tau}\omega^2)\in \mathcal{C}_{g}^{\beta}([-T,T];R^d),\quad a.s.
\end{equation*}
for any $T>0,\tau\in R$, where $\theta_{\tau}\omega^2$ represents the second order process of the path $\theta_{\tau}\omega^1$.       As in \cite{MR3680943} and
\cite{MR3912728}, for any $\delta\in (0,1)$, we define a random variable
$\mathcal{G}_{\delta}: \Omega \rightarrow \mathbb{R}^{d}$
\begin{equation*}
\mathcal{G}_{\delta}(\omega^1)=\frac{1}{\delta} \omega^1(\delta).
\end{equation*}
Then we have
\begin{equation*}
	\mathcal{G}_{\delta}\left(\theta_{t}
	\omega^1\right)=\frac{1}{\delta}(\omega^1(t+\delta)-\omega^1(t)),\quad t\in R.
\end{equation*}
By the the properties of the fractional Brownian motion, it follows that
$\mathcal{G}_\delta(\theta_{t}\omega)$ is a stochastic process on $(\Omega,\mathcal{F},\mathbb{P}_{H},\theta)$.
Let
\begin{equation*}
W_{\delta}(t,\omega^1):=\int_{0}^{t}\mathcal{G}_{\delta}(\theta_{s}\omega^1)ds,\quad t\in R.
\end{equation*}
$W_{\delta}(t,\omega^1)$ may be viewed as an approximation of the
fractional Brownian motion, and it is a stochastic process  with stationary increments on $(\Omega,\mathcal{F},\mathbb{P}_{H},\theta)$.  Furthermore,
\begin{equation*}
\mathbb{W}_{\delta}(\omega^1)_{s,t}:=\int_{s}^{t}W_{\delta}(\cdot,\omega^1)_{s,r}\otimes
dW_{\delta}(r,\omega^1), \quad s,t\in R
\end{equation*}
is a smooth second order process  on $(\Omega,\mathcal{F},\mathbb{P}_{H},\theta)$. It is well defined as a Riemann-Stieljes
integral. For $\tau \in \mathbb{R}, s \leq t \in \mathbb{R}, \omega \in \Omega$, we also define
\begin{equation*}
\theta_{\tau}\mathbb{W}_{\delta}(\omega^1)_{s,t}:=\int_{s}^{t}\left(\theta_{\tau} W_{\delta}(r,\omega^1)-\theta_{\tau} W_{\delta}(s,\omega^1)\right) \otimes d \theta_{\tau}W_{\delta}(r,\omega^1).
\end{equation*}
Then it is easy to see that
\begin{equation*}
	\theta_{\tau}\mathbb{W}_{\delta}(\omega^1)_{s,t}=\mathbb{W}_{\delta}(\theta_{\tau}\omega^1)_{s,t}=\mathbb{W}_{\delta}(\omega^1)(s+\tau, t+\tau) .
\end{equation*}
Thus, we will consider the convergence of the smooth rough paths
$\boldsymbol{\omega}_{\delta}=\left(W_{\delta}(\cdot,\omega^{1}), \mathbb{W}_{\delta}(\omega^{1})\right)$
generated by the approximative process $W_{\delta}(\cdot,\omega^{1})$.

We shall use the following result to compare the distance of two
different rough paths, see \cite[Theorem 3.3]{MR4174393}.
\begin{lemma}\label{theorem 4.1}
Let $q^{\prime}\geq 2$, $\frac{1}{2\rho}>\frac{1}{q^{\prime}}$. Assume, for all
$s,t\in[-T,T]$ and some constant $C>0$ both $(X,\mathbb{X})$ and $(Y,\mathbb{Y})$
satisfy
\begin{equation*}
\begin{aligned}
&|X_{s,t}|_{L^{q^{\prime}}}\leq C|t-s|^{\frac{1}{2\rho}}\quad
|\mathbb{X}_{s,t}|_{L^{\frac{q^{\prime}}{2}}}\leq C|t-s|^{\frac{1}{\rho}}
\quad\text{and}\\
&|Y_{s,t}|_{L^{q^{\prime}}}\leq C|t-s|^{\frac{1}{2\rho}}\quad
|\mathbb{Y}_{s,t}|_{L^{\frac{q^{\prime}}{2}}}\leq C|t-s|^{\frac{1}{\rho}}.
\end{aligned}
\end{equation*}
Further, let
\begin{equation*}
\triangle X :=X-Y, \quad\triangle \mathbb{X}=\mathbb{X}-\mathbb{Y}.
\end{equation*}
For some $\epsilon>0$ and $s,t\in [-T,T]$, we have the following relation
\begin{equation*}
\begin{aligned}
|\triangle X_{s,t}|_{L^{q^{\prime}}}\leq C\epsilon |t-s|^{\frac{1}{2\rho}},\quad
|\triangle \mathbb{X}_{s,t}|_{\frac{q^{\prime}}{2}}\leq C\epsilon
|t-s|^{\frac{1}{\rho}}.
\end{aligned}
\end{equation*}
Then there exists a constant $M>0$ depending on $C$, such that
\begin{equation*}
	|\interleave\triangle X\interleave_{\alpha}|_{L^{q^{\prime}}}\leq
	M\epsilon,\quad
	|\interleave\triangle\mathbb{X}\interleave_{2\alpha}|_{L^{\frac{q^{\prime}}{2}}}\leq
	M\epsilon.
\end{equation*}
Furthermore, if $\frac{1}{2\rho}-\frac{1}{q^{\prime}}>\frac{1}{3}$ then, for each $\alpha\in(\frac{1}{3},\frac{1}{2\rho}-\frac{1}{q^{\prime}})$ we have
\begin{equation*}
	\interleave\mathbf{X}\interleave_{\alpha},\interleave\mathbf{Y}\interleave_{\alpha}\in
	L^{q^{\prime}}\quad \text{and}\quad |\rho_{\alpha,[-T,T]}(\mathbf{X},
	\mathbf{Y})|_{L^{\frac{q^{\prime}}{2}}}\leq M\epsilon.
\end{equation*}
\end{lemma}
\begin{remark}
The $L^{q^{\prime}}$ and $L^{\frac{q^{\prime}}{2}}$-estimates in the above theorem
 can be derived from the $L^{2}$ estimates. The reason
for that is that we consider a Gaussian process,  which moment norms of
each order are equivalent in chaos spaces\cite[Theorem 2.7.2,
hypercontractivity ]{MR2962301}.
\end{remark}
Before establishing the approximation of a fractional
Brownian rough path by a smooth path, we first present the following lemma.
\begin{lemma}\label{theorem 4.2}
For a geometric  fractional Brownian  rough path $\boldsymbol{\omega}:=(\omega^{1},\omega^{2}) $
with Hurst index
$H\in(\frac{1}{3},\frac{1}{2}]$, we have the following estimates
\begin{equation*}
\begin{aligned}
&|\omega_{t}^{1,i}-\omega^{1,i}_{s}|_{L^{q^{\prime}}}\leq
C(T,q^{\prime},H)(t-s)^{\beta^{\prime}},\\
&|\omega_{s,t}^{2,i,j}|_{L^{\frac{q^{\prime}}{2}}}\leq
C(T,q^{\prime},H,\beta^{\prime})(t-s)^{2\beta^{\prime}},
\end{aligned}
\end{equation*}
where $\frac{1}{3}<\beta^{\prime}<H$,   $q^{\prime}\geq 2$, $1\leq i,j\leq d$
and $0\leq s<t\leq T$.
\end{lemma}
\begin{proof}
The proof of this lemma is similar to Brownian motion case\cite[Theorem 4.5]{MR4266219}, using Lemma \ref{lemma 3.1} and Lemma \ref{lemma 3.2} to complete the proof.
\end{proof}
 Lemma \ref{theorem 4.2} shows that the geometric
rough  $\boldsymbol{\omega}=(\omega^{1},\omega^{2})$ satisfies the condition
which we have formulated in Lemma \ref{theorem 4.1} for
$\rho=\frac{1}{2\beta^{\prime}}$. Next, we need to check  that the
approximative process satisfies the same conditions.
\begin{lemma}\label{theorem 4.3}
Let $\boldsymbol{\omega}:=(\omega^{1},\omega^{2}) $ be a geometric fractional Brownian rough path   with Hurst index $H\in(\frac{1}{3},\frac{1}{2}]$, its approximation $\boldsymbol{\omega}_{\delta}=(W_{\delta}(\cdot,\omega^{1}),\mathbb{W}_{\delta}(\omega^{1}))$ forms a smooth rough path and satisfies the following estimates
\begin{equation*}
\begin{aligned}
&|W^{i}_{\delta}(t,\omega^{1})-W^{i}_{\delta}(s,\omega^{1})|_{L^{q^{\prime}}}\leq C(T,q^{\prime},H,\beta^{\prime})(t-s)^{\beta^{\prime}},\\
&|\mathbb{W}^{i,j}_{\delta}(\omega^{1})_{s,t}|_{L^{\frac{q^{\prime}}{2}}}\leq C(T,q^{\prime},H,\beta^{\prime})(t-s)^{2\beta^{\prime}},
\end{aligned}
\end{equation*}
where $\frac{1}{3}<\beta^{\prime}<H$, $q^{\prime}\geq 2$, $1\leq i,j\leq d$ and $0\leq s<t\leq T$.
\end{lemma}
\begin{proof}
For each $\delta\in(0,1]$, $\mathbb{E}\exp(iuW^{j}_{\delta}(t,\omega^{1}))=\exp(-\frac{1}{2}u^{2}\frac{t^{2}}{\delta^{2-2H}})$, and since every linear combination of $(W^{j}_{\delta}(t_{1},\omega^{1}),\cdots,W^{j}_{\delta}(t_{k},\omega^{1})),k\in\mathbb{Z}^{+}$ has a univariate Gaussian distribution, then $(W^{j}_{\delta}(t_{1},\omega^{1}),\cdots,W^{j}_{\delta}(t_{k},\omega^{1})),k\in\mathbb{Z}^{+}$ is a multivariate Gaussian random variable.  Thus $W_{\delta}(t,\omega^{1})$ is a Gaussian process. Moreover, for any $s,t\in\mathbb{R}$, we have
  \begin{equation*}
  	\begin{aligned}
 W^{i}_{\delta}(t+s,\omega^{1})- W^{i}_{\delta}(s,\omega^{1})&=\frac{1}{\delta}\int_{0}^{t+s}\omega^{1,i}_{r+\delta}-\omega^{1,i}_{r}dr-\frac{1}{\delta}\int_{0}^{s}\omega^{1,i}_{r+\delta}-\omega^{1,i}_{r}dr\nonumber\\
 &=\frac{1}{\delta}\int_{s}^{t+s}\omega^{1,i}_{r+\delta}-\omega^{1,i}_{r}dr\nonumber\\
 &=\frac{1}{\delta}\int_{0}^{t}\omega^{1,i}_{r+\delta+s}-\omega^{1,i}_{r+s}dr\nonumber\\
 &=\frac{1}{\delta}\int_{0}^{t}\theta_{s}\omega^{1,i}_{r+\delta}-\theta_{s}\omega^{1,i}_{r}dr=W^{i}_{\delta}(t,\theta_{s}\omega^{1})\nonumber.
 \end{aligned}
 \end{equation*}
 This equation implies that the increment $W^{i}_{\delta}(t+s,\omega^{1})- W^{i}_{\delta}(s,\omega^{1})$ has the same distribution as $W^{i}_{\delta}(t,\omega^{1})$, since the $\theta_{t}$-invariance for $\mathbb{P}_{H}$, namely, $\theta_{s}\omega^{1}$ has the same distribution as $\omega^{1}$. Hence $W^{i}_{\delta}(t,\omega^{1})$ is a Gaussian process with stationary increments.
 It is obvious that $\boldsymbol{\omega}_{\delta}=(W_{\delta}(\cdot,\omega^{1}),\mathbb{W}_{\delta}(\omega^{1}))$ forms a smooth rough path and it means that $\boldsymbol{\omega}_{\delta}$ is also a geometric rough path. It remains to check the above estimates.
For the first component, we set $W^{i}_{\delta}(\cdot,\omega^{1}):=W_{\delta}(\cdot,\omega^{1,i})$, if $\delta\geq t-s>0$, then by  H\"{o}lder's inequality we have
\begin{equation}\label{4.4}
\begin{aligned}
\mathbb{E}|W^{i}_{\delta}(t,\omega^{1})-W^{i}_{\delta}(s,\omega^{1})|^{2}&=\mathbb{E}\left|\int_{s}^{t}\frac{1}{\delta}(\omega^{1,i}_{s+\delta}-\omega^{1,i}_{s})ds\right|^{2}\\
&\leq\frac{t-s}{\delta^{2}}\int_{s}^{t}\mathbb{E}|\omega^{1,i}(r+\delta)-\omega^{1,i}(r)|^{2}ds\\
&\leq C(t-s)^{2}\frac{1}{\delta^{2-2H}}\leq C(t-s)^{2H}\\
&\leq C(t-s)^{2\beta^{\prime}}.
\end{aligned}
\end{equation}
Moreover, for $0<\delta<t-s$, using  H\"{o}lder's inequality, we have
\begin{equation}\label{4.5}
\begin{aligned}
&\mathbb{E}|W^{i}_{\delta}(t,\omega^{1})-W^{i}_{\delta}(s,\omega^{1})|^{2}=\frac{1}{\delta^{2}}\mathbb{E}\int_{s}^{t}\int_{s}^{t}(\omega^{1,i}_{\delta+s_{1}}-\omega^{1,i}_{s_{1}})(\omega^{1,i}_{\delta+s_{2}}-\omega^{1,i}_{s_{2}})ds_{1}ds_{2}\\
&~~=\frac{1}{\delta^{2}}\int_{s}^{t}\int_{s}^{t}R_{\omega^{1,i}}\left(\begin{array}{cc}
s_{1} & s_{1}+\delta \\
s_{2} & s_{2}+\delta
\end{array}\right)ds_{1}ds_{2}\\
&~~\leq \frac{1}{\delta^{2}}\left(\int_{s}^{t}\int_{s}^{t}\left|R_{\omega^{1,i}}{}\left(\begin{array}{cc}
s_{1} & s_{1}+\delta \\
s_{2} & s_{2}+\delta
\end{array}\right)\right|^{\frac{1}{2H}}ds_{1}ds_{2}\right)^{2H}(t-s)^{2(1-2H)}.
\end{aligned}
\end{equation}
For fixed $\delta<t-s$, let $|\mathcal{P}(s,t)|<\delta^{\frac{1}{2H}}$ and $|\mathcal{P}^{\prime}(s,t)|<\delta^{\frac{1}{2H}}$, then by the definition of a Riemann integral and the $\frac{1}{2H}$-variation of  the rectangular increment of the covariance  of a fractional Brownian motion. We have
\begin{equation}\label{4.6}
\begin{aligned}
&\frac{1}{\delta^{2}}\left(\int_{s}^{t}\int_{s}^{t}\left|R_{\omega^{1,i}}{}\left(\begin{array}{cc}
s_{1} & s_{1}+\delta \\
s_{2} & s_{2}+\delta
\end{array}\right)\right|^{\frac{1}{2H}}ds_{1}ds_{2}\right)^{2H}\\
&=\!\left(\!\lim_{|\mathcal{P}(\!s,t\!)|\!\vee\!|\mathcal{P}^{\prime}(s,t)|\rightarrow\! 0}\!\sum_{\substack{[\!t_{i},t_{i+1}\!]\in\mathcal{P}(\!s,t\!)\\ [\!t^{\prime}_{j},t_{j+1}\!]\in\mathcal{P}^{\prime}(s,t) }}\!\left|\!R_{\omega^{1,i}}{}\!\!\left(\begin{array}{cc}
\!t_{i} & t_{i}\!+\!\delta \\
\!t^{\prime}_{j} & t^{\prime}_{j}\!+\!\delta
\end{array}\!\right)\!\right|^{\frac{1}{2H}}\!\!\!\!\frac{t_{i+1}\!-\!t_{i}}{\delta^{\frac{1}{2H}}}\!\frac{t^{\prime}_{j+1}\!-\!t^\prime_{j}}{\delta^{\frac{1}{2H}}}\!\right)^{2H}\\
&\leq \left(\sup_{\substack{\mathcal{P}(s,t)\\ \mathcal{P}^{\prime}(s,t)}}\sum_{\substack{[t_{i},t_{i+1}]\in\mathcal{P}(s,t)\\ [t^{\prime}_{j},t_{j+1}]\in\mathcal{P}^{\prime}(s,t) }}\left|R_{\omega^{1,i}}{}\left(\begin{array}{cc}
t_{i} & t_{i}+\delta \\
t^{\prime}_{j} & t^{\prime}_{j}+\delta
\end{array}\right)\right|^{\frac{1}{2H}}\!\right)^{2H}\\
&\leq \|R_{\omega^{1,i}}\|_{\frac{1}{2H}-var;[s,t+\delta]^{2}}\leq C(t+\delta-s)^{2H}\leq C(H)(t-s)^{2H},
\end{aligned}
\end{equation}
where $\delta<t-s$, it follows that the constant $C(H)$ does not depend on $\delta$. Put \eqref{4.6} into \eqref{4.5}, we have
 \begin{equation}\label{4.7}
 \mathbb{E}|W^{i}_{\delta}(t,\omega^{1})-W^{i}_{\delta}(s,\omega^{1})|^{2}\leq  C(T,H,\beta^{\prime})(t-s)^{2\beta^{\prime}}.
 \end{equation}
Together \eqref{4.4}, \eqref{4.7} and applying the hypercontractivity of first order chaos, we obtain
\begin{equation}\label{4.8}
|W^{i}_{\delta}(t,\omega^{1})-W^{i}_{\delta}(s,\omega^{1})|_{L^{q^{\prime}}}\leq C(T,q^{\prime},H,\beta^{\prime})(t-s)^{\beta^{\prime}}.\end{equation}
 For the second component $\mathbb{W}_{\delta}(\omega^{1})$, we first compute $\|R_{W^{m}_{\delta}(\cdot,\omega^{1})}\|^{\frac{1}{2H}}_{\frac{1}{2H};[s,t]^{2}}$ for $1\leq m\leq d$. For $0<\delta< t- s$ and any partition $\mathcal{P}(s,t),\mathcal{P}^{\prime}(s,t)$, similar to the above calculation we have that
 \begin{equation}\label{4.9}
 \begin{aligned}
&~~\sum_{\substack{[t_{i},t_{i+1}]\in\mathcal{P}\\ [t^{\prime}_{j},t^{\prime}_{j+1}]\in\mathcal{P}^{\prime}}}|
\mathbb{E}[W^{m}_{\delta}(\cdot,\omega^{1})_{t_{i},t_{i+1}}W^{m}_{\delta}(\cdot,\omega^{1})_{t^{\prime}_{j},t^{\prime}_{j+1}}]|^{\frac{1}{2H}}\\
&~~=\sum_{\substack{[t_{i},t_{i+1}]\in\mathcal{P}\\ [t^{\prime}_{j},t^{\prime}_{j+1}]\in\mathcal{P}^{\prime}}}
\left|\mathbb{E}\int_{t_{i}}^{t_{i+1}}\int_{t_{j}^{\prime}}^{t^{\prime}_{j+1}}\frac{1}{\delta^{2}}\theta_{s_{1}}\omega^{1,m}(\delta)\theta_{s_{2}}\omega^{1,m}(\delta)ds_{1}ds_{2}\right|^{\frac{1}{2H}}\\
&~~=\sum_{\substack{[t_{i},t_{i+1}]\in\mathcal{P}\\ [t^{\prime}_{j},t^{\prime}_{j+1}]\in\mathcal{P}^{\prime}}}
\left(\int_{t_{i}}^{t_{i+1}}\int_{t^{\prime}_{j}}^{t^{\prime}_{j+1}}\frac{1}{\delta^{2}}\left|R_{\omega^{1,m}}{}\left(\begin{array}{cc}
s_{1} & s_{1}+\delta \\
s_{2} & s_{2}+\delta
\end{array}\right)\right|ds_{1}ds_{2}\right)^{\frac{1}{2H}}\\
&~~\leq \left(\int_{s}^{t}\int_{s}^{t}\frac{1}{\delta^{2}}\left|R_{\omega^{1,m}}{}\left(\begin{array}{cc}
s_{1} & s_{1}+\delta \\
s_{2} & s_{2}+\delta
\end{array}\right)\right|ds_{1}ds_{2}\right)^{\frac{1}{2H}}\\
&~~\leq\frac{1}{\delta^{\frac{1}{H}}}\left(\int_{s}^{t}\int_{s}^{t}\left|R_{\omega^{1,m}}{}\left(\begin{array}{cc}
s_{1} & s_{1}+\delta \\
s_{2} & s_{2}+\delta
\end{array}\right)\right|^{\frac{1}{2H}}ds_{1}ds_{2}\right)(t-s)^{2(\frac{1}{2H}-1)},
 \end{aligned}
\end{equation}
 where the last inequality holds by H\"{o}lder's inequality.  Choosing $|\mathcal{P}(s,t)|\bigvee|\mathcal{P}^{\prime}(s,t)|\leq \delta^{\frac{1}{2H}}$, then we have
 \begin{equation}\label{4.10}
 \begin{aligned}
& \frac{1}{\delta^{\frac{1}{H}}}\left(\int_{s}^{t}\int_{s}^{t}\left|R_{\omega^{1,m}}{}\left(\begin{array}{cc}
s_{1} & s_{2}+\delta \\
s_{2} & s_{2}+\delta
\end{array}\right)\right|^{\frac{1}{2H}}ds_{1}ds_{2}\right)\\
&=\lim_{|\substack{\mathcal{P}(s,t)|\rightarrow 0\\ |\mathcal{P}^{\prime}(s,t)|\rightarrow 0}}\frac{1}{\delta^{\frac{1}{H}}}\!\sum_{\substack{[t_{i},t_{i+1}]\in\mathcal{P}(s,t)\\ [t^{\prime}_{j},t^{\prime}_{j+1}]\in\mathcal{P}^{\prime}(s,t)}}\!\left|R_{\omega^{1,m}}{}\left(\begin{array}{cc}
t_{i} & t_{i}+\delta \\
t^{\prime}_{j} & t^{\prime}_{j}+\delta
\end{array}\right)\!\right|^{\frac{1}{2H}}\!(t_{i+1}\!-\!t_{i})(t^{\prime}_{j+1}\!-\!t^{\prime}_{j})\\
&\leq \sup_{\substack{\mathcal{P}(s,t)\\ \mathcal{P}^{\prime}(s,t)}} \sum_{\substack{[t_{i},t_{i+1}]\in\mathcal{P}(s,t)\\ [t^{\prime}_{j},t^{\prime}_{j+1}]\in\mathcal{P}^{\prime}(s,t)}}\left|R_{\omega^{1,m}}{}\left(\begin{array}{cc}
t_{i} & t_{i}+\delta \\
t^{\prime}_{j} & t^{\prime}_{j}+\delta
\end{array}\right)\right|^{\frac{1}{2H}}\\
&\leq \|R_{\omega^{1,m}}\|^{\frac{1}{2H}}_{\frac{1}{2H}-var;[s,t+\delta]^{2}}\leq C(H)(t-s).
 \end{aligned}
 \end{equation}
 Putting \eqref{4.10} into \eqref{4.9} we have that
\begin{equation}\label{4.11}
\|R_{W^{m}_{\delta}(\cdot,\omega^{1})}\|^{\frac{1}{2H}}_{\frac{1}{2H}-var;[s,t]^{2}}\leq C(T,H)(t-s).\end{equation}
For $\delta>t-s>0$, similar to \eqref{4.4}, we have
 \begin{equation}\label{4.11*}
\|R_{W^{m}_{\delta}(\cdot,\omega^{1})}\|^{\frac{1}{2H}}_{\frac{1}{2H}-var;[s,t]^{2}}\leq C(T,H)(t-s).
 \end{equation}
Using the elementary inequality $ (\sum|a_{i}|^{\frac{1}{2\beta^{\prime}}})^{2\beta^{\prime}}\leq  (\sum|a_{i}|^{\frac{1}{2H}})^{2H}$(\cite[Proposition 5.3]{MR2604669}), we obtain
\begin{equation}\label{4.12}\|R_{W^{m}_{\delta}(\cdot,\omega^{1})}\|^{\frac{1}{2\beta^{\prime}}}_{\frac{1}{2\beta^{\prime}}-var;[s,t]^{2}}\leq\|R_{W^{m}_{\delta}(\cdot,\omega^{1})}\|^{\frac{1}{2H}}_{\frac{1}{2H}-var;[s,t]^{2}}\leq C(T,H)(t-s).\end{equation}
 For the case $i\neq j$, using hypercontractivty, Lemma \ref{lemma 3.1} and \eqref{4.12} we derive
 \begin{equation}\label{4.13}
 \begin{aligned}
 |\mathbb{W}^{i,j}_{\delta}(\omega^{1})_{s,t}|_{L^{\frac{q^{\prime}}{2}}}&\leq C(q^{\prime}) |\mathbb{W}^{i,j}_{\delta}(\omega^{1})_{s,t}|_{L^{2}}\\
 &\leq C(q^{\prime},T,H)\sqrt{\|R_{W^{i}_{\delta}(\cdot,\omega^{1})}\|_{\frac{1}{2\beta^{\prime}};[s,t]^{2}}\|R_{W^{j}_{\delta}(\cdot,\omega^{1})}\|_{\frac{1}{2\beta^{\prime}};[s,t]^{2}}}\\
 &\leq C(q^{\prime},T,H,\beta^{\prime})(t-s)^{2\beta^{\prime}}.
 \end{aligned}
 \end{equation}
 For the case $i=j$, using the property of a geometric rough path and \eqref{4.7}, we obtain
 \begin{equation}\label{4.14}
 \begin{aligned}
 \mathbb{E}(\mathbb{W}^{i,i}_{\delta}(\omega^{1}))_{t,s})^{2}\leq \frac{1}{4}\mathbb{E}((W^{i}_{\delta}(\cdot,\omega^{1}))_{s,t})^{2}\leq C(T,H,\beta^{\prime})(t-s)^{2\beta^{\prime}}.
 \end{aligned}
 \end{equation}
 By hypercontractivity, we have
 \begin{equation}\label{4.15}
 |\mathbb{W}^{i,i}_{\delta}(\omega^{1})_{s,t}|_{L^{\frac{q^{\prime}}{2}}}\leq C(T,q^{\prime},H,\beta^{\prime})(t-s)^{2\beta^{\prime}}.
 \end{equation}
\end{proof}
Furthermore, we consider the difference process $X_{\delta}$  between a fractional Brownian motion and its approximative process. We can establish the following theorem.
\begin{lemma}\label{theorem 4.4}
Let $\boldsymbol{\omega}:=(\omega^{1},\omega^{2}) $ be a fractional Brownian rough path with Hurst index $H\in(\frac{1}{3},\frac{1}{2}]$.  Consider  the approximation $\boldsymbol{\omega}_{\delta}=(W_{\delta}(\cdot,\omega^{1}),\mathbb{W}_{\delta}(\omega^{1}))$ and define $X_{\delta}(t):=\omega^{1}_{t}-W_{\delta}(t,\omega^{1})$. Then  $X_{\delta}(t)$ generates a geometric rough path $(X_{\delta},\mathbb{X}_{\delta})$ and we have the following estimates
\begin{equation*}
\begin{aligned}
&|X^{i}_{\delta}(t)-X^{i}_{\delta}(s)|_{L^{q^{\prime}}}\leq C\delta^{H-\beta^{\prime}}(t-s)^{\beta^{\prime}},\\
&|\mathbb{X}^{i,j}_{\delta,s,t}|_{L^{\frac{q^{\prime}}{2}}}\leq C(H,q^{\prime},\beta^{\prime})\delta^{2H-2\beta^{\prime}}(t-s)^{2\beta^{\prime}},
\end{aligned}
\end{equation*}
where $\frac{1}{3}<\beta^{\prime}<H$ and $q^{\prime}\geq 2$, $1\leq i,j\leq d$, $0\leq s<t\leq T$ .
\end{lemma}
\begin{proof}
$X_{\delta}(t)$ is a  Gaussian process with stationary increments as $W_{\delta}(t,\omega^{1})$. We first compute $\sigma^{2}_{X^{i}_{\delta}}(t-s)=\mathbb{E}|X_{\delta}^{i}(t)-X_{\delta}^{i}(s)|^{2}$. Similar to Lemma \ref{theorem 4.3}, the computation is divided into two steps. For $0<\delta\leq t-s$, using triangle inequality, stationary increments of the fractional Brownian motion and Cauchy-Schwartz inequality, we have
\begin{equation}
\begin{aligned}
\mathbb{E}|&X_{\delta}^{i}(t)-X_{\delta}^{i}(s)|^{2}=\mathbb{E}\left|\frac{1}{\delta}\int_{s}^{t}\omega^{1,i}(r+\delta)-\omega^{1,i}(r)dr-\omega^{1,i}(t)+\omega^{1,i}(s)\right|^{2}\\
&\leq 2\mathbb{E}\left|\frac{1}{\delta}\int_{t}^{t+\delta}\omega^{1,i}(r)-\omega^{1,i}(t)dr\right|^{2}\!+\! 2\mathbb{E}\left|\frac{1}{\delta}\int_{s}^{s+\delta}\omega^{1,i}(r)-\omega^{1,i}(s)dr\right|^{2}\\
&\leq 2\frac{1}{\delta}\int_{t}^{t+\delta}\mathbb{E}|\omega^{1,i}(r)-\omega^{1,i}(t)|^{2}dr+2\frac{1}{\delta}\int_{s}^{s+\delta}\mathbb{E}|\omega^{1,i}(r)-\omega^{1,i}(s)|^{2}dr\\
&\leq 2\frac{1}{\delta}\left(\int_{t}^{t+\delta}(r-t)^{2H}dr+\int_{s}^{s+\delta}(r-s)^{2H}dr\right)\\
&\leq C\delta^{2H}\leq C\delta^{2H-2\beta^{\prime}}(t-s)^{2\beta^{\prime}}.
\end{aligned}
\end{equation}
For $\delta>t-s>0$, we have
\begin{equation}
	\begin{aligned}
\mathbb{E}|X_{\delta}^{i}(t)-X_{\delta}^{i}(s)|^{2}&=\mathbb{E}\left|\frac{1}{\delta}\int_{s}^{t}\omega^{1,i}(r+\delta)-\omega^{1,i}(r)dr-\omega^{1,i}(t)+\omega^{1,i}(s)\right|^{2}\\
&\leq \frac{2}{\delta^{2}}\mathbb{E}\left|\int_{s}^{t}\omega^{1,i}(r+\delta)-\omega^{1,i}(r)dr\right|^{2}+2|t-s|^{2H}\\
&\leq \frac{2}{\delta^{2}}(t-s)^{2}\delta^{2H}+2|t-s|^{2H}\\
&\leq 2(t-s)^{2H}(1+\frac{(t-s)^{2-2H}}{\delta^{2-2H}})\leq 4\delta^{2H-2\beta^{\prime}}(t-s)^{2\beta^{\prime}}.
\end{aligned}
\end{equation}
Hence, we have $\sigma^{2}_{X_{\delta}^{i}}(t-s)\leq C\delta^{2H-2\beta^{\prime}}(t-s)^{2\beta^{\prime}}$. Furthermore, by  hypercontractivity, we obtain
\begin{equation}\label{4.18}
|X^{i}_{\delta}(t)-X^{i}_{\delta}(s)|_{L^{q^{\prime}}}\leq C(q^{\prime})\delta^{H-\beta^{\prime}}(t-s)^{\beta^{\prime}}.\end{equation}
In order to illustrate that  $X_{\delta}$ forms a rough path, by Theorem \ref{theorem 3.1}, we need to compute $\|R_{X_{\delta}^{i}}\|_{\rho-var;[s,t]^{2}}^{\rho}$ for $\rho=\frac{1}{2\beta^{\prime}}$.  The Cauchy-Schwartz inequality and \eqref{4.18} for $q^{\prime}=2$, yield
\begin{equation}\label{4.19}
	\begin{aligned}
&\left\|R_{X^{i}_{\delta}}\right\|_{\rho-var;[s,t]^{2}}^{\rho}\!\!=\!\!\sup_{\substack{\mathcal{P}(s,t)\\ \mathcal{P}^{\prime}(s,t)}}\!\sum_{\substack{[u,\! v]\in\mathcal{P}(s,\!t)\\ [u^{\prime}\!,\! v^{\prime}]\in\mathcal{P}^{\prime}(s,\!t)}}\!\left|\mathbb{E}[(X_{\delta}^{i}(v)\!-\!X_{\delta}^{i}(u))(X_{\delta}^{i}(v^{\prime})\!-\!X_{\delta}^{i}(u^{\prime}))]\right|^{\rho}\\
&\leq \sup_{\substack{\mathcal{P}(s,t)\\ \mathcal{P}^{\prime}(s,t)}}\sum_{\substack{[u,v]\in\mathcal{P}(s,t)\\ [u^{\prime},v^{\prime}]\in\mathcal{P}^{\prime}(s,t)}} \left|\left(\mathbb{E}[X^{i}_{\delta}(v)-X^{i}_{\delta}(u)]^{2}\right)^{\frac{1}{2}}\left(\mathbb{E}[X^{i}_{\delta}(v^{\prime})-X^{i}_{\delta}(u^{\prime})]^{2}\right)^{\frac{1}{2}}\right|^{\rho}\\
&\leq C(\rho) \sup_{\substack{\mathcal{P}(s,t)\\ \mathcal{P}^{\prime}(s,t)}}\sum_{\substack{[u,v]\in\mathcal{P}(s,t)\\ [u^{\prime},v^{\prime}]\in\mathcal{P}^{\prime}(s,t)}}\left(\delta^{2H-2\beta^{\prime}}(v-u)^{\beta^{\prime}}(v^{\prime}-u^{\prime})^{\beta^{\prime}}\right)^{\rho}\\
&\leq C(\rho) \sup_{\substack{\mathcal{P}(s,t)\\ \mathcal{P}^{\prime}(s,t)}}\sum_{\substack{[u,v]\in\mathcal{P}(s,t)\\ [u^{\prime},v^{\prime}]\in\mathcal{P}^{\prime}(s,t)}}\left(\delta^{2H-2\beta^{\prime}}\frac{(v-u)^{2\beta^{\prime}}+(v^{\prime}-u^{\prime})^{2\beta^{\prime}}}{2^{\beta^{\prime}}}\right)^{\rho}\\
&\leq C(\rho)2^{\rho-1}\delta^{(2H-2\beta^{\prime})\rho}\sup_{\substack{\mathcal{P}(s,t)\\ \mathcal{P}^{\prime}(s,t)}} \sum_{\substack{[u,v]\in\mathcal{P}(s,t)\\ [u^{\prime},v^{\prime}]\in\mathcal{P}^{\prime}(s,t)}}\left(\frac{(v-u)+(v^{\prime}-u^{\prime})}{2^{\beta^{\prime}\rho}}\right)\\
&\leq C(\rho)\delta^{(2H-2\beta^{\prime})\rho}(t-s).
\end{aligned}
\end{equation}
Hence, we have $\|R_{X_{\delta}^{i}}\|_{\frac{1}{2\beta^{\prime}}-var;[s,t]^{2}}\leq C(\beta^{\prime})\delta^{2H-2\beta^{\prime}}|t-s|^{2\beta^{\prime}}$. Furthermore, on account of  Theorem \ref{theorem 3.1}, we know that
$X_{\delta}(t)$  forms a geometric rough path which satisfies the following estimates:
\begin{description}
  \item[1.] For $i\neq j$, using Lemma \ref{lemma 3.1} and hypercontractivity,
  \begin{equation}\label{4.20}
  \begin{aligned}
  \left|\mathbb{X}^{i,j}_{\delta,s,t}\right|_{\frac{q^{\prime}}{2}}\leq C(q^{\prime})\left|\mathbb{X}^{i,j}_{\delta,s,t}\right|_{2}\leq C(q^{\prime},\beta^{\prime})\delta^{2H-2\beta^{\prime}}(t-s)^{2\beta^{\prime}}.
  \end{aligned}
  \end{equation}
  \item[2.] For $i=j$, by the property of geometric rough path and hepercontractivity,
   \begin{equation}\label{4.21}
  \begin{aligned}\left|\mathbb{X}^{i,i}_{\delta,s,t}\right|_{\frac{q^{\prime}}{2}}\leq C(q^{\prime})\left|\mathbb{X}^{i,i}_{\delta,s,t}\right|_{2}= C(q^{\prime})\delta^{2H-2\beta^{\prime}}(t-s)^{2\beta^{\prime}}.
  \end{aligned}
  \end{equation}
\end{description}
\end{proof}
\begin{remark}
In the proof of  Lemma \ref{theorem 4.2}, the condition of Lemma \ref{lemma 3.2}
holds for the fractional Brownian motion. However, we can not apply Lemma \ref{lemma 3.2} to the stochastic processes  $X^{i}_{\delta}(\cdot)=\omega^{1,i}_{\cdot}-W^{i}_{\delta}(\cdot,\omega^{1})$ and $W^{i}_{\delta}(\cdot,\omega^{1})$, since the concavity of
$\sigma^{2}_{X^{i}_{\delta}}(u)$ and
$\sigma^{2}_{W^{i}_{\delta}(\cdot,\omega^{1})}(u)$ is too complex to check, it is mainly due to the complex structure
of the functions
$\sigma^{2}_{X^{i}_{\delta}}(u),\sigma^{2}_{W^{i}_{\delta}(\cdot,\omega^{1})}(u)$,
see for $H=\frac{1}{2}$ in \cite{MR4266219}.  For our
considerations concavity of $\sigma_{X^{i}_{\delta}}^{2}(u),\sigma^{2}_{W^{i}_{\delta}(\cdot,\omega^{1})}(u)$
are not necessary. We only use the properties of
$\|R_{X^{i}_{\delta}}\|_{\frac{1}{2\beta^{\prime}}-\!var;[s,t]^{2}},\|R_{W^{i}_{\delta}(\cdot,\omega^{1})}\|_{\frac{1}{2\beta^{\prime}}\!-\!var;[s,t]^{2}}$ in the proof of Lemma \ref{theorem 4.3}, \ref{theorem 4.4}.
\end{remark}
Finally, we need to complete the estimate of $\mathbf{\omega}_{s,t}^{2,i,j}\!-\!\mathbb{W}^{i,j}_{\delta}(\theta_{\cdot}\omega^{1})_{s,t},1\leq i,j\leq d$. To this end, our idea is based on that $(X_{\delta},\mathbb{X}_{\delta})$ can be regarded as a translation of a Gaussian rough path $(\omega^{1},\omega^{2})$ in direction $-W^{i}_{\delta}(\cdot,\omega^{1})$(see \cite[(11.5), page 188]{MR4174393} ), namely, the second order process $\mathbb{X}_{\delta}$ is generated by the shifted path $\omega^{1}-W^{i}_{\delta}(\cdot,\omega^{1})$.  Thus, we have the following theorem.
\begin{theorem}\label{theorem 4.5}
Let $\delta\in (0,1]$, $\beta^{\prime}\in (\frac{1}{3},H)$ and $q^{\prime}\geq 2$ such that $\beta^{\prime}-\frac{1}{q^{\prime}}>\frac{1}{3}$. Then for each $\beta\in\left(\frac{1}{3},\beta^{\prime}-\frac{1}{q^{\prime}}\right)$, we have $\interleave \boldsymbol{\omega}\interleave_{\beta}$ and $\interleave\boldsymbol{\omega}_{\delta}\interleave_{\beta}\in L^{q^{\prime}}$. Moreover, there exists a positive constant $C(q^{\prime},\beta^{\prime}, H ,T)$ such that
\begin{equation*}
\left|\rho_{\beta,s,t}(\boldsymbol{\omega}_{\delta},\boldsymbol{\omega})\right|_{L^{q^{\prime}}}\leq C(q^{\prime},\beta^{\prime},H,T)\delta^{H-\beta^{\prime}},\quad \text{for} -T\leq s<t\leq T.
\end{equation*}
Therefore,
\begin{equation*}
\lim_{\delta\rightarrow 0}\left|\rho_{\beta,s,t}(\boldsymbol{\omega}_{\delta},\boldsymbol{\omega})\right|_{L^{\frac{q^{\prime}}{2}}}=0.
\end{equation*}
\end{theorem}
\begin{proof}
For $i\neq j$ we have the following splitting on $\mathbb{X}^{i,j}_{\delta,s,t}$:
\begin{equation*}
\begin{aligned}
\mathbb{X}^{i,j}_{\delta,s,t}\!=\!\mathbf{\omega}^{2,i,j}_{s,t}\!\!-\!\!\!\int_{s}^{t}\!\!\!\left(\!W^{i}_{\delta}(\cdot,\omega^{1})\right)_{s,r}\!d\omega^{1,j}_{r}\!\!-\!\!\!\int_{s}^{t}\!\!\!\omega^{1,i}_{s,r}dW^{j}_{\delta}(r,\omega^{1})\!+\!\!\!\int_{s}^{t}\!\!\!\left(W^{i}_{\delta}(\cdot,\omega^{1})\right)_{s,r}\!dW^{j}_{\delta}(r,\omega^{1}),
\end{aligned}
\end{equation*}
where since $\omega^{1,i}\in C^{\beta},W^{i}_{\delta}(\cdot,\omega^{1}),i\in{1,2,\cdots,d}$ are $C^{1}$-smooth and $1+\beta>1$, then the last three integrals above are  Young integrals\cite{MR2604669}, the reader can find this decomposition in \cite[page 188]{MR4174393}. 
Hence, from the above computation we obtain
\begin{equation*}
\begin{aligned}
\mathbf{\omega}_{s,t}^{2,i,j}\!-\!\mathbb{W}^{i,j}_{\delta}(\theta_{\cdot}\omega^{1})_{s,t}\!=\!\mathbb{X}^{i,j}_{\delta,s,t}\!+\!\int_{s}^{t}\left(W^{i}_{\delta}(\cdot,\omega^{1})\right)_{s,r}dX^{j}_{\delta}(r)+\int_{s}^{t}X^{i}_{\delta,s,r}dW^{j}_{\delta}(r,\omega^{1}).
\end{aligned}
\end{equation*}
By Lemma \ref{lemma 3.1}, \eqref{4.12}, \eqref{4.19}, we have
\begin{equation}\label{4.25}
\begin{aligned}
\mathbb{E}\left|\int_{s}^{t}\left(W^{i}_{\delta}(\cdot,\omega^{1})\right)_{s,r}dX^{j}_{\delta}(r)\right|^{2}&\!\leq\! \left\|R_{X^{j}_{\delta}}\!\right\|_{\frac{1}{2\beta^{\prime}}-var;[s,t]^{2}}\!\left\|R_{W^{i}_{\delta}(\cdot,\omega^{1})}\!\right\|_{\frac{1}{2\beta^{\prime}}-var;[s,t]^{2}}\\
&\leq C(T,\beta^{\prime},H)\delta^{2H-2\beta^{\prime}}\left(t-s\right)^{4\beta^{\prime}}
\end{aligned}
\end{equation}
and
\begin{equation}\label{4.26}
\begin{aligned}
\mathbb{E}\left|\int_{s}^{t}X^{i}_{\delta,s,r}dW^{j}_{\delta}(r,\omega^{1})\right|^{2}\leq C(T,\beta^{\prime},H)\delta^{2H-2\beta^{\prime}}\left(t-s\right)^{4\beta^{\prime}}.
\end{aligned}
\end{equation}
Combining  \eqref{4.20}, \eqref{4.25}, \eqref{4.26}  and hypercontractivity, we have
\begin{equation}\label{4.27}
\left|\mathbf{\omega}_{s,t}^{2,i,j}\!-\!\mathbb{W}^{i,j}_{\delta}(\theta_{\cdot}\omega^{1})_{s,t}\right|_{L^{\frac{q^{\prime}}{2}}} \leq C(q^{\prime},T,H,\beta^{\prime})\delta^{H-\beta^{\prime}} (t-s)^{2\beta^{\prime}}.
\end{equation}
 Let $\epsilon=\delta^{H-\beta^{\prime}}$, $\rho=\frac{1}{2\beta^{\prime}}$ in Lemma \ref{theorem 4.1}, and the constant $C$ only depend on $\beta^{\prime}, T, H$, but not on $\delta$.
For $i=j$, using the property of a geometric rough path, we have
\begin{equation}\label{4.28}
\begin{aligned}
&\mathbb{E}\left(\omega^{2,i,i}_{s,t}-\mathbb{W}^{i,i}_{\delta}(\cdot,\omega^{1})_{s,t}\right)^{2}\!=\!\frac{1}{4}\mathbb{E}\left(\left(W^{i}_{\delta}(t,\omega^{1})\!-\!W^{i}_{\delta}(s,\omega^{1})\right)^{2}\!-\!\left(\omega^{1}_{t}\!-\!\omega^{1}_{s}\right)^{2}\right)^{2}\\
&~~~~~~=\frac{1}{4}\mathbb{E}\left(\left(W^{i}_{\delta}(t,\omega^{1})-W^{i}_{\delta}(s,\omega^{1})+\omega^{i}_{t}-\omega^{i}_{s}\right)\left(X_{\delta}^{i}(t)-X_{\delta}^{i}(s)
\right)\right)^{2}\\
&~~~~~~\leq \frac{1}{4}\left(\mathbb{E}\left(W^{i}_{\delta}(t,\omega^{1})-W^{i}_{\delta}(s,\omega^{1})+\omega^{i}_{t}-\omega^{i}_{s}\right)^{4}\mathbb{E}\left(X_{\delta}^{i}(t)-X_{\delta}^{i}(s)\right)^{4}\right)^{\frac{1}{2}}.
\end{aligned}
\end{equation}
By  Lemma  \ref{theorem 4.2} and \eqref{4.8} we can estimate
	\begin{equation}\label{4.29}
		\begin{aligned}
&\mathbb{E}\left|W^{i}_{\delta}(t,\omega^{1})\!-\!W^{i}_{\delta}(s,\omega^{1})\!+\!\omega^{i}_{t}\!-\!\omega^{i}_{s}\right|^{4}\\
&~~~~~~~\leq \mathbb{E}\left|W^{i}_{\delta}(t,\omega^{1})\!-\!W^{i}_{\delta}(s,\omega^{1})\right|^{4}\!+\!8\mathbb{E}\left|\omega^{1,i}_{t}-\omega^{1,i}_{s}\right|^{4}\\
&~~~~~~~\leq C(t-s)^{4\beta^{\prime}},
     \end{aligned}
 \end{equation}
where $C$ is uniform with respect to $\delta\in(0,1]$ and depends on $T, H, \beta^{\prime}$. Combining \eqref{4.28}-\eqref{4.29} and  \eqref{4.21} we obtain
\begin{equation}
	\begin{aligned}
\left|\mathbf{\omega}_{s,t}^{2,i,i}\!-\!\mathbb{W}^{i,i}_{\delta}(\theta_{\cdot}\omega^{1})_{s,t}\right|_{L^{2}}\!&=\!\frac{1}{2}\left(\mathbb{E}\!\left(\!\left(W^{i}_{\delta}(t,\omega^{1})\!-\!W^{i}_{\delta}(s,\omega^{1})\!\right)^{2}\!-\!\left(\omega^{1}_{t}\!-\!\omega^{1}_{s}\right)^{2}\!\right)^{2}\!\right)^{\frac{1}{2}}\\
&\leq C\delta^{H-\beta^{\prime}}|t-s|^{2\beta^{\prime}}.
    \end{aligned}
\end{equation}
 Applying the hypercontractivity to second order chaos we can get the $L^{\frac{q^{\prime}}{2}}$-norm,
 \begin{equation}\label{4.32}
 \left|\mathbf{\omega}_{s,t}^{2,i,i}\!-\!\mathbb{W}^{i,i}_{\delta}(\theta_{\cdot}\omega^{1})_{s,t}\right|_{L^{\frac{q^{\prime}}{2}}} \leq C(q^{\prime},T,H,\beta^{\prime})\delta^{H-\beta^{\prime}} (t-s)^{2\beta^{\prime}}.
 \end{equation}
 By  \eqref{4.27}, \eqref{4.32}, Lemma \ref{theorem 4.2}, Lemma \ref{theorem 4.3}, then Lemma \ref{theorem 4.1} can be applied, for any $\beta\in(\frac{1}{3},\frac{1}{2})$. We can choose $\beta^{\prime}>\beta$  and  $q^{\prime}\geq2$ such that $\beta^{\prime}-\frac{1}{q^{\prime}}>\frac{1}{3}$  and we have $\interleave\boldsymbol{\omega}\interleave_{\beta}$, $\interleave\boldsymbol{\omega}_{\delta}\interleave_{\beta}\in L^{q^{\prime}}$. Furthermore, we have
\begin{equation*}
\left|\rho_{\beta, s, t}\left(\boldsymbol{\omega}_{\delta}, \boldsymbol{\omega}\right)\right|_{L^{\frac{q^\prime}{2}}} \leq C(q^{\prime}, \beta^{\prime}, H,T) \delta^{H-\beta^{\prime}}.
\end{equation*}
 Thus,
\begin{equation*}
 \lim_{\delta\rightarrow 0}\left|\rho_{\beta, s, t}\left(\boldsymbol{\omega}_{\delta}, \boldsymbol{\omega}\right)\right|_{L^{\frac{q^\prime}{2}}}=0.
 \end{equation*}
  Due to Kolmogorov's test criteria for rough paths\cite[Theorem 3.3, Proposition 15.24]{MR4266219,MR2604669}, it is necessary to require $\beta<\beta^{\prime}$.  Note that our consideration only  Lemma \ref{theorem 4.2}, \ref{theorem 4.3}, \ref{theorem 4.4}, Theorem \ref{theorem 4.5} on  $[0, T]$. For our purpose we need to extend the previous results  to $[-T,T]$.  We can extend $\boldsymbol{\omega}$ from $[0,T]$ to $[-T,T]$. Indeed, by Chen's identity,  for $s<0<t\in[-T,T]$ and $1\leq i,j\leq d$ we have
  	 $$\omega^{2,i,j}_{s,t}:=\omega^{2,i,j}_{s,0}+\omega^{2,i,j}_{0,t}+\omega^{1,i}(s)\omega^{i,j}(t),$$
where the definition of $\omega^{2,i,j}_{s,0}$ need to be checked for $s<0$. For $i\neq j$, we consider the definition of $\omega^{2,i,j}_{s,t}$ by an integral as follows
\begin{equation*}
\begin{aligned}
\omega^{2,i,j}_{s,0}&:=\int_{s}^{0}(\omega^{1,i}_{r}-\omega^{1,i}_{s})d\omega^{1,j}_{r}=\int_{0}^{-s}\theta_{s}\omega^{1,i}(r)d\theta_{s}\omega^{1,j}(r)\\
&=\lim_{|\mathcal{P}(0,-s)|\rightarrow 0}\sum_{[u,v]\in\mathcal{P}(0,-s)}(\omega^{1,i}_{u+s}-\omega^{1,i}_{s})(\omega^{1,j}_{v+s}-\omega^{1,j}_{u+s}),
\end{aligned}
\end{equation*}
where $\theta_{s}\omega^{1,i}$ is also a fractional Brownian motion  with the same rectangular increments of the covariance as  $\omega^{1,i}$ or $i\in\{1,\cdots ,d\}$. Using Lemma \ref{lemma 3.1},  Lemma \ref{lemma 3.2}, we can prove that the limit exists in $L^{2}$ sense as  Lemma \ref{theorem 4.2}, thus $\omega^{2,i,j}_{s,t}$ exists for $s<t\in \mathbb{R}$. For $i=j$, define
 \begin{equation*}
 \omega^{2,i,i}_{s,0}:=\frac{1}{2}(\omega^{1,i,i}_{s}-\omega^{1,i,i}_{0})^{2}.
 \end{equation*}
For Lemma \ref{theorem 4.3} and Lemma \ref{theorem 4.4}, we use the same method to extend the temporal area from $[0,T]$ to $[-T,T]$, thus we  obtain the convergence on $[-T,T]$, the proof is completed.
\end{proof}
Furthermore, we have the following theorem
\begin{theorem}\label{theorem 4.6}
Let $\boldsymbol{\omega}=(\omega^{1},\omega^{2})$ be the canonical lift of the fractional Brownian motion and $\boldsymbol{\omega}_{\delta}=\left(W_{\delta}(\cdot,\omega^{1}),\mathbb{W}_{\delta}(\cdot,\omega^{1})\right)$ be the approximation of $\boldsymbol{\omega}$. Then we have
\begin{equation*}
\rho_{\beta,-T,T}(\boldsymbol{\omega},\boldsymbol{\omega}_{\delta})\rightarrow 0,\quad \text{as}\quad \delta\rightarrow 0
\end{equation*}
for any  $T>0$, $\beta\in(\frac{1}{3},\frac{1}{2})$.  Furthermore, the convergence takes place for all $\omega$ in a $\theta$-invariant set $\Omega^{\prime}$ of full measure.
\end{theorem}
\begin{proof}
We divide  the proof of this theorem into two steps. Our idea is to find a sequence $\{\delta_{i}\}_{i\in\mathbb{N}}$ converging sufficiently fast to zero and prove that $\rho_{\beta,-T,T}(\boldsymbol{\omega}_{\delta_{i}}, \boldsymbol{\omega})\rightarrow 0$ takes place  in a $\theta$-invariant set $\Omega^{\prime}$ of full measure as $i\rightarrow \infty$, and then we need to illustrate the convergence relation between $\boldsymbol{\omega}_{\delta}$ and $\boldsymbol{\omega}_{\delta_{i}}$. Note that we only need to prove the convergence relation for $T=n,n\in\mathbb{N}$.

For the first step, namely  $\rho_{\beta,-n,n}(\boldsymbol{\omega}_{\delta_{i}}, \boldsymbol{\omega})\rightarrow 0$ as $i\rightarrow \infty$, the proof is similar to \cite[Theorem 4.6]{MR4266219}, we only give the outline of the proof. Choosing $
\delta_{i}=i^{-\frac{4}{q^\prime(H-\beta^\prime)}}$. According to Theorem \ref{theorem 4.5} and Chebyshev's inequality, we obtain
\begin{equation*}
\mathbb{P}_{H}(\rho_{\beta,-n,n}(\boldsymbol{\omega}_{\delta_{i}},\boldsymbol{\omega})>\epsilon)\leq \frac{C(q^\prime,\beta^\prime,H,n)}{\epsilon^{q^{\prime}}}i^{-4}
\end{equation*}
 for any $\epsilon>0$. In particular, we  choose $\epsilon=\left(\frac{i}{2}\right)^{-\frac{1}{q^{\prime}}}$. Thus, the Borel-Cantelli lemma shows that there exists a set of  full measure $\Omega^{(n)}\subset \Omega$  and $i_{0}(\omega,n)$ for $\omega\in \Omega^{(n)}$ such that
 \begin{equation*}
 \rho_{\beta,-n,n}(\boldsymbol{\omega}_{\delta_{i}},\boldsymbol{\omega})\leq \left(\frac{i}{2}\right)^{-\frac{1}{q^{\prime}}}
 \end{equation*}
 as $i\geq i_{0}(\omega,n)$. Let $\hat{\Omega}^{0}=\cap_{n\geq1} \Omega^{(n)}$, then $\mathbb{P}_{H}(\hat{\Omega}^{0})=1$. In addition, replacing $\omega$ by $\theta_{\tau}\omega$ we introduce a set of  full measure $\hat{\Omega}^{\tau},\tau\in R$.  For simplicity, we only consider the argument of $\theta$-invariance  for the second order process.  Assume $\tau,q, s,t\in R,\omega^{1}\in \Omega$. Since $\boldsymbol{\omega}_{\delta_i}$ is the smooth approximation, we have
\begin{equation*}
 \theta_{\tau}\mathbb{W}_{\delta_{i}}(\cdot,\omega^{1})_{s,t}=\mathbb{W}_{\delta_{i}}(\cdot,\omega^{1})_{s+\tau,t+\tau}.
 \end{equation*}
 Then
\begin{equation*}
 \theta_{\tau+q}\mathbb{W}_{\delta_{i}}(\cdot,\omega^{1})_{s,t}=\theta_{\tau}\mathbb{W}_{\delta_{i}}(\cdot,\omega^{1})_{s+q,t+q},
\end{equation*}
 let $\delta_{i}\rightarrow 0$ as $i\rightarrow \infty$, we obtain
  \begin{equation*}
  \theta_{\tau+q}\omega^{2}_{s,t}=\theta_{\tau}\omega^{2}_{s+q,t+q}
  \end{equation*}
  and $\hat{\Omega}^{\tau}\subset \hat{\Omega}^{\tau+q}$. In addition, we have
 \begin{equation*}
  \theta_{\tau}\omega^{2}_{s,t}=\theta_{\tau+q}\omega^{2}_{s-q,t-q},
  \end{equation*}
  which shows that $\hat{\Omega}^{\tau+q}\subset \hat{\Omega}^{\tau}$. Thus we have $\hat{\Omega}^{\tau}=\hat{\Omega}^{0},\tau\in R$. Then for any $q\in R$, $\theta_{-q}\hat{\Omega}^{0}=\theta^{-1}_{q}\hat{\Omega}^{0}=\theta^{-1}_{q}\hat{\Omega}^{q}=\hat{\Omega}^{0}$. Hence $\Omega^{\prime}=\hat{\Omega}^{0}$.

  For the second step, namely, $\rho_{\beta,-n,n} (\boldsymbol{\omega}_{\delta},\boldsymbol{\omega}_{\delta_i})\rightarrow 0$  for each fixed   $\omega^1\in\Omega^{\prime}$ as $\delta\rightarrow 0$. We note that there exists a $i=i(\delta)$ such that $\delta\in(\delta_{i+1},\delta_{i}]$ for each $\delta\in(0,1)$. We only consider the second process, the path component can be studied similarly.  For any $s<t\in[-n,n]$, we have
  \begin{align*}
   \mathbb{W}_{\delta}(\cdot,\omega^{1})_{s,t}\!\!-\!\!\mathbb{W}_{\delta_{i}}(\cdot,\omega^{1})_{s,t}&\!=\!\frac{1}{\delta^{2}}\int_{s}^{t}\int_{s}^{r}\theta_{r^{\prime}}
   \omega^{1}_\delta\theta_{r}
   \omega^{1}_\delta dr^{\prime}dr\!-\!\frac{1}{\delta_{i}^{2}}\int_{s}^{t}\int_{s}^{r}\theta_{r^{\prime}}
   \omega^{1}_{\delta_{i}}\theta_{r}
   \omega^{1}_{\delta_i} dr^{\prime}dr\\
   =&\left(\frac{1}{\delta^2}-\frac{1}{\delta_{i}^2}\right)\int_{s}^{t}\int_{s}^{r}\theta_{r^{\prime}}
   \omega^{1}_\delta\theta_{r}
   \omega^{1}_\delta dr^{\prime}dr\\
   +&\frac{1}{\delta_{i}^2}\int_{s}^{t}\int_{s}^{r}\theta_{r^{\prime}}
   \omega^{1}_{\delta_{i}}\theta_{r}
   \omega^{1}_{\delta_i}-\theta_{r^{\prime}}\omega^{1}_\delta\theta_{r}
   \omega^{1}_\delta dr^{\prime}dr\\
=:& I_{s,t}^{1}+I_{s,t}^{2}.
  \end{align*}
  By the H\"{o}lder regularity of the path, we obtain
\begin{equation}\label{3.28*}
	\begin{aligned}
   I^{1}_{s,t}&=\left(\frac{\delta_{i}^2-\delta^2}{\delta_{i}^2\delta^2}\right)\int_{s}^{t}\int_{s}^{r}(\omega^1_{\delta+r^\prime}-\omega^1_{r^\prime})(\omega^1_{\delta+r}-\omega^1_{r})dr^\prime dr,\\
   &\leq \left(\frac{(\delta_{i}^2-\delta^2)\interleave \omega^1\interleave_{\beta,[-n,n]}^2}{\delta_{i}^2\delta^2}\right)\int_{s}^{t}\int_{s}^{r}\delta^2dr^\prime dr\\
   &\leq C(n)(t-s)^2\frac{(\delta_{i}^2-\delta^2)\delta^{2\beta}}{\delta_{i}^2\delta^2}.
   \end{aligned}
   \end{equation}
Similarly, we have
   \begin{equation}\label{3.29*}
   I_{s,t}^{2}\leq C(n)(t-s)^{2}\frac{(\delta_{i}-\delta)^{\beta}\delta_{i}^{\beta}}{\delta_{i}^{2}}.
   \end{equation}
By the definition of $\delta_{i}$ and $\delta\in (\delta_{i+1},\delta_{i}]$,  then
  \begin{equation}
  \frac{\delta_{i}^{2}-\delta^{2}}{\delta^{2}}\leq \frac{i^{-\frac{8}{q^\prime(H-\beta^\prime)}}-(i+1)^{-\frac{8}{q^{\prime}(H-\beta^\prime)}}}{(i+1)^{-\frac{8}{q^\prime(H-\beta^\prime)}}}=\left(1+\frac{1}{i}\right)^{\frac{8}{q^\prime(H-\beta^\prime)}}-1.
  \end{equation}
  Let $f(x)=(1+x)^{\alpha}-Cx-1$, where $x\in(0,1], \alpha>0$. It is sufficient to show that $f(x)\leq 0$ if the constant $C>\alpha 2^{\alpha-1}$. Indeed, since $C>\alpha 2^{\alpha-1}$, then the derivative $f^{\prime}(x)\leq 0$. Thus, $f(x)\leq f(0)=0$.  Let $x=\frac{1}{i}$ and $\alpha=\frac{8}{q^\prime(H-\beta^\prime)}$. Hence, there exists a constant $C>\frac{8}{q^\prime(H-\beta^\prime)}2^{\frac{8}{q^\prime(H-\beta^\prime)}-1}$ such that
  \begin{equation}\label{3.31*}\frac{\delta_{i}^2-\delta_{i+1}^2}{\delta^2_{i+1}}\leq C\frac{1}{i},\end{equation}
  similarly, the above constant can guarantee that the inequality
 \begin{equation}\label{3.32*}\frac{\delta_i-\delta}{\delta_i}\leq C\frac{1}{i}\end{equation}
  holds. So using the inequalities \eqref{3.28*}-\eqref{3.29*} and \eqref{3.31*}-\eqref{3.32*}, we obtain
  \begin{equation*}
  I_{s,t}^{1}\leq C(n)(t-s)^2i^{\frac{8-8\beta}{(H-\beta^{\prime})q^{\prime}}-1} \quad \text{and} \quad I_{s,t}^{2}\leq C(n)(t-s)^2i^{\frac{8-8\beta}{(H-\beta^{\prime})q^{\prime}}-\beta},
  \end{equation*}
  then according to Theorem \ref{theorem 4.5}, choosing  $q^{\prime}>\max\left\{\frac{8-8\beta}{(H-\beta^{\prime})\beta},\frac{1}{\beta^{\prime}-\beta}\right\}$, we obtain
  \begin{equation*}
  \interleave \mathbb{W}_{\delta}(\cdot,\omega^{1})-\mathbb{W}_{\delta_{i}}(\cdot,\omega^{1})\interleave_{2\beta,[-n,n]}\rightarrow 0 \quad \text {as} \quad \delta\rightarrow 0.
  \end{equation*}
  We complete the proof of this  theorem.
\end{proof}
\section{Wong-Zakai approximation of the rough differential equation}
In this section, we  consider  the following rough differential equation driven by a rough path via the canonical lift(see \cite[page 156]{MR2604669}) of fractional Brownian motion
\begin{equation}\label{5.1}
dy=(Ay+f(y)dt+g(y)d\boldsymbol{w}
\end{equation}
and its approximation form
\begin{equation}\label{5.2}
dy^{\delta}=(Ay^{\delta}+f(y^{\delta}))dt+g(y^{\delta})d\boldsymbol{w}_{\delta}
\end{equation}
	with initial data $x$ and $x^{\delta}$ respectively, where we assume  $A\in \mathbb{R}^{m\times m}$, $f:\mathbb{R}^{m}\rightarrow \mathbb{R}^{m}$, $g:\mathbb{R}^{m}\rightarrow \mathbb{R}^{m\times d}$, and the driving path $\boldsymbol{\omega}\in \mathcal{C}^{\beta}([0,T];\mathbb{R}^{d})\subset \mathcal{C}^{p-var}([0,T];\mathbb{R}^{d})$, with $\beta\in(\frac{1}{3},\frac{1}{2})$ and $p=\frac{1}{\beta}$. $\boldsymbol{\omega}_{\delta}$ is defined in Section \ref{section-intersection}.  For the rough differential equation \eqref{5.1}, Duc established the existence and uniqueness result \cite[Theorem 2.1]{duc2020asymptotic,MR4385780}   in the Gubinelli sense,  Riedel and Scheutzow\cite{MR3567487} achieved results for solutions of \eqref{5.1} in the sense of Friz-Victoir\cite{MR2604669}. Although Friz and Hairer\cite{MR4174393} constructed the theory of rough differential equations, their stability of solutions with rough noise can not be applied here. Indeed,  \cite[Theorem 3.1, 4.3]{MR3567487} required that the drift term is locally Lipschitz  and linear growth,  however, the  diffusion term $g(y)$ is $C_b^{\gamma}(R^m),\gamma>3$, namely $D^ig(y),i\in{0,1,2,3}$ and  $(\gamma-3)$-H\"{o}lder semi-norm of $D^3g(y)$ are uniform bounded. Friz and Victoir\cite[Theorem 12.10]{MR2604669} imposed  conditions that the drift term is at least differentiable and its derivative is bounded. Friz and Hairer\cite[page 141]{MR4174393} adopted the same method as \cite{MR2604669}, namely, the time variable as the component of the path $(t,\omega_{t})$, it means that  \cite[Theorem 8.5]{MR2604669} can be applied. However, it  requires that the drift term is three times continuously differentiable and all derivatives are bounded. Based on the work\cite{duc2020asymptotic,MR4385780},  we want to get the Wong-Zakai approximation of the solution under  framework of  controlled rough paths.  Compared with these results,  our conditions are weaker, namely we assume :
\begin{description}
                                                                         \item[\textbf{H1:}] $f:\mathbb{R}^{m}\rightarrow\mathbb{R}^{m}$ is globally Lipschitz continuous with Lipschitz constant $C_{f}$, the function $g\in C^{3}_{b}$, namely, it is  three times continuously differentiable and all derivative are bounded. Moreover, let
                                                                           \begin{equation*}
                                                                           C_{g}:=max\left\{\|g\|_{\infty},\|Dg\|_{\infty},\|D^{2}g\|_{\infty},\|D^{3}g\|_{\infty}\right\}.
                                                                           \end{equation*}

                                                                           \end{description}
\begin{remark}\label{remark 5.1}
 We can consider a general nonlinear $f(u)$ instead of $Au+f(u)$ for the well-posedness  of a solution.
In addition,   \eqref{5.2} can be interpreted as a (random) non-autonomous dynamical system defined by an ordinary differential equation with coefficients which are Lipschitz
continuous in the state variable and continuous in the time variable. Hence this equation has a unique global solution for every initial condition. This solution coincides with the solution in the rough path sense. Since the approximate noise is smooth, then the Gubinelli derivative is not unique. However, the path component $y^\delta$ is unique. Indeed, let $(y^\delta_1,(y^\delta_1)^\prime)$ and $(y^\delta_2,(y^\delta_2)^\prime)$ be the solution of \eqref{5.2}, then $y^\delta_1$ and $y^\delta_2$ are the solution of non-autonomous ordinary differential equation driven by a smooth path. Hence, the uniqueness of non-autonomous ordinary differential equation shows that $y^\delta_1=y^\delta_2$. Then we choose a specific Gubinelli derivative $g(y^{\delta})$ in this paper. So we use rough paths theory to prove the existence and uniqueness of \eqref{5.2}.  Thus, the solution $y^\delta$ of equation \eqref{5.2} also generates a random dynamical system.
\end{remark}
\subsection{ Controlled rough path and rough integral}
\begin{definiton}\label{def 5.1}
The path $y\in C^{\beta}(I,R^{m}),\beta\in(\frac{1}{3},\frac{1}{2})$ is called controlled by $\omega$ if there exist $y^{\prime}\in C^{\beta}(I,\mathbb{R}^{m\times d})$ and $R^{y}\in C^{2\beta}(I^{2},\mathbb{R}^{m})$ such that
\begin{equation}
\begin{aligned}
y_{s,t}=y^{\prime}_{s}\omega_{s,t}+R^{y}_{s,t}
\end{aligned}
\end{equation}
for all $s<t\in I$. $y^{\prime}$ is called Gubinelli derivative of $y$ and $R^{y}$ is the remainder term.
\end{definiton}

We denote by  $\mathcal{D}_{\omega}^{2\beta}(I,R^{m})$ the set of all $(y,y^{\prime})$ which are controlled by $\omega$ and equipped with the norm
\begin{equation*}
\|y,y^{\prime}\|_{\omega,2\beta}=\|y_{T_{1}}\|+\|y^{\prime}_{T_{1}}\|+\interleave y^{\prime} \interleave_{\alpha,I}+\interleave R^{y}\interleave_{2\alpha,I}.
\end{equation*}
Then $\mathcal{D}_{\omega}^{2\beta}(I,R^{m})$ is a Banach space \cite[page 70]{MR4174393}. Note  that in general the Gubinelli derivative is not uniquely determined, but the condition of truly rough\cite[page 109]{MR4174393} can guarantee the uniqueness of the Gubinelli derivative. In our paper the fractional Brownian motion is truly rough.
\begin{remark}\label{Rem4.2}
The Gubinelli derivative of the Riemann integral $\int_{0}^t Ay_{r}+f(y_{r})dr$ in this paper is $0$, namely $\left(\int_0^t Ay_{r}+f(y_{r})dr,0\right)\in \mathcal{D}_{\omega}^{2\beta}(I,R^{m})$, it is easy to check this fact by the definition of $\mathcal{D}_{\omega}^{2\beta}(I,R^{m})$. So
$Ay+f(y)$ is contained in the remainder term  $R^{y}$.
\end{remark}
For the composition of a smooth function  and a rough path we have
\begin{lemma}
Let $(y,y^{\prime})\in \mathcal{D}_{\omega}^{2\beta}(I,R^{m})$ and $g\in C^{2}_{b}(R^{m},R^{m\times d})$. Then $g(y)$ is also controlled by $\omega$, where
\begin{equation}
\begin{aligned}
\left((g(y))_{t}^{\prime},R^{g(y)}_{s,t}\right)=\left(Dg(y_{t})y_{t}^{\prime},g(y_{t})-g(y_{s})-Dg(y_{s})y_{s}^{\prime}\omega_{s,t}\right).
\end{aligned}
\end{equation}
\end{lemma}
The proof of the Lemma can be found in \cite[Lemma 7.3]{MR4174393}.

Now we can use the controlled rough path to define a rough integral. 
Hence  based on the Sewing Lemma\cite[Lemma 4.2]{MR4174393}, the rough integral can be defined,
and there is a constant $C_{\beta}>1$ such that
\begin{equation}\label{5.8}
\begin{aligned}
&\left|\int_{s}^{t}y_{r}d\boldsymbol{\omega}_{r}-y_{s}\omega^{1}_{s,t}-y_{s}^{\prime}\omega^{2}_{s,t}\right|\\
\leq &C_{\beta}(t-s)^{3\beta}(\interleave\omega^{1}\interleave_{\beta,[s,t]}\interleave R^{y}\interleave_{2\beta,[s,t]^{2}}+\interleave y^{\prime} \interleave_{\beta,[s,t]}\interleave \omega^{2} \interleave_{2\beta,[s,t]^{2}}).
\end{aligned}
\end{equation}
Furthermore, we  consider the above results under the $p$-variation norm. We define the $p$-variation norm of the control rough path as follows
\begin{equation}
\begin{aligned}
\left\|(y,y^{\prime})\right\|_{\omega,p,I}=\|y_{T_{1}}\|+\|y^{\prime}_{T_{1}}\|+\interleave y^{\prime}\interleave_{p-var,I}+\interleave R^{y}\interleave_{q-var,I^{2}}.
\end{aligned}
\end{equation}
Then there exists a $C_{p}>1$ such that \eqref{5.8} can be replaced by
\begin{equation}\label{5.10}
\begin{aligned}
&\left|\int_{s}^{t}y_{r}d\boldsymbol{\omega}_{r}-y_{s}\omega^{1}_{s,t}-y_{s}^{\prime}\omega^{2}_{s,t}\right|\\
\leq &C_{p}\left(\interleave\omega^{1}\interleave_{p-var,[s,t]}\interleave R^{y}\interleave_{q-var,[s,t]^{2}}+\interleave y^{\prime} \interleave_{p-var,[s,t]}\interleave \omega^{2} \interleave_{q-var,[s,t]^{2}}\right).
\end{aligned}
\end{equation}
It is directly obtained by \cite[Lemma 6.2]{MR2604669} and \eqref{5.8}.
\subsection{Existence and uniqueness theorem}
Firstly, we introduce a  sequence of stopping times which is presented in \cite{MR3112937}. Consider  $\frac{1}{p}\in(\frac{1}{3},\frac{1}{2})$ and $\eta\in(0,1)$, we define a  sequence of stopping time as follow
$$\tau_{0}=T_{1},\quad \tau_{i+1}:=\inf\{t>\tau_{i}:\interleave \boldsymbol{\omega}\interleave_{p-var,[\tau_{i},t]}=\eta\}\wedge T_{2}.$$
Then $N_{\eta,I,p}(\boldsymbol{\omega})$ is defined by $\sup\{i\in \mathbb{N}:\tau_{i}\leq T_{2}\}$. It is clear that
$$N_{\eta,I,p}(\boldsymbol{\omega})\leq 1+\eta^{-p}\interleave\boldsymbol{\omega} \interleave^{p}_{p-var,I}.$$
Indeed, we can obtain the estimate by Lemma \ref{Lemma 2.1}, namely
$$(N_{\eta,I,p}(\boldsymbol{\omega})-1)\eta^{p}=\sum_{k=0}^{N_{\eta,I,p}(\boldsymbol{\omega})-2}\interleave \boldsymbol{\omega}\interleave^{p}_{p-var,[\tau_{k},\tau_{k+1}]}\leq \interleave \boldsymbol{\omega}\interleave^{p}_{p-var,I}. $$
Based on the sequence of stopping times,  the solution of equation \eqref{5.1} can be constructed and has the following estimates
\begin{lemma}\label{lemma 5.3}
Let $\eta=\frac{1}{4C_{p}C_{g}}$, Then there exists a unique solution $(y,g(y))$ for \eqref{5.1} with any initial data on the interval $[\tau,\tau+T],\tau\in R,T>0$, and have the following estimates
\begin{equation*}
\begin{aligned}
\|y\|_{\infty,[\tau,\tau+T]}\leq \left[\|y_{\tau}\|+\left(\frac{f(0)}{L}+\frac{1}{C_{p}}\right)N_{\frac{1}{4C_{p}C_{g}},[\tau,\tau+T],p}(\boldsymbol{\omega})\right]e^{4LT},
\end{aligned}
\end{equation*}
\begin{equation*}
\begin{aligned}
&\|y_{\tau}\|+\interleave y,R^{y} \interleave_{p-var,[\tau,\tau+T]}\nonumber\\
\leq&\left[\|y_{\tau}\|+\left(\frac{f(0)}{L}+\frac{1}{C_{p}}\right)N_{\frac{1}{4C_{p}C_{g}},[\tau,\tau+T],p}(\boldsymbol{\omega})\right]e^{4LT}N^{\frac{p-1}{p}}_{\frac{1}{4C_{p}C_{g}},[\tau,\tau+T],p}(\boldsymbol{\omega}),
\end{aligned}
\end{equation*}
where $L=\|A\|+C_{f}$ and $\interleave y,R^{y}\interleave_{p-var,[s,t]}:=\interleave y \interleave_{p-var,[s,t]}+\interleave R^{y}\interleave_{q-var,[s,t]^{2}}.$
\end{lemma}
The proof of this lemma  is similar to  \cite[Theorem 3.8]{duc2020asymptotic} and \cite{MR4385780}. 
Furthermore, for the approximated system \eqref{5.2} we have the following result.
\begin{lemma}\label{theorem 5.1}
Let $\eta=\frac{1}{4C_{p}C_{g}}$, Then there exists a unique solution $(y^\delta,g(y^\delta))$ for \eqref{5.2} with any initial data on the interval $[\tau,\tau+T],\tau\in R,T>0$, and have the following estimates
\begin{equation*}
\begin{aligned}
\|y^{\delta}\|_{\infty,[\tau,\tau+T]}\leq \left[\|y^{\delta}_{\tau}\|+\left(\frac{f(0)}{L}+\frac{1}{C_{p}}\right)N_{\frac{1}{4C_{p}C_{g}},[\tau,\tau+T],p}(\boldsymbol{\omega_{\delta}})\right]e^{4LT},
\end{aligned}
\end{equation*}
\begin{equation*}
\begin{aligned}
&\|y^{\delta}_{\tau}\|+\interleave y^{\delta},R^{y^{\delta}} \interleave_{p-var,[\tau,T]}\\
\leq&\left[\|y^{\delta}_{\tau}\|+\left(\frac{f(0)}{L}+\frac{1}{C_{p}}\right)N_{\frac{1}{4C_{p}C_{g}},[\tau,\tau+T],p}(\boldsymbol{\omega_{\delta}})\right]e^{4LT}N^{\frac{p-1}{p}}_{\frac{1}{4C_{p}C_{g}},[\tau,\tau+T],p}(\boldsymbol{\omega_{\delta}}),
\end{aligned}
\end{equation*}
where $L=\|A\|+C_{f}$ and $\interleave y^{\delta},R^{y^{\delta}}\interleave_{p-var,[s,t]}:=\interleave y^{\delta} \interleave_{p-var,[s,t]}+\interleave R^{y^{\delta}}\interleave_{q-var,[s,t]^{2}}.$
\end{lemma}
\begin{remark}
Note that the norm of $y^{\prime}=g(y)$ and $(y^{\delta})^{\prime}=g(y^{\delta})$  can be estimated by the norm of $y$ and $y^{\delta}$ respectively. Thus, we don't consider the semi-norm $\|y,y^{\prime}\|_{p-var,[s,t]}=\|y^{\prime}\|_{p-var,[s,t]}+\|R^{y}\|_{q-var,[s,t]^{2}}$. In addition, by non-uniqueness of the Gubinelli derivative for system \eqref{5.2}, we could let $(y^{\delta})^{\prime}=0$, this means that $y^{\delta}_{s,t}=R^{y^{\delta}}_{s,t}$, but it brings some additional problems for the approximation of the solution.
\end{remark}
\subsection{Smooth and stationary Wong-Zakai approximation for the solution}
In this subsection we shall use the smooth and stationary  Wong-Zakai approximation  for a geometric fractional rough path in Section \ref{approxiamtion of  the noise} to complete the approximation of the solution. Based on the rough path theory, we  get the approximation of the solution in some local intervals,  and the stopping times technique help us  complete the approximation  of the solution in any finite interval.
\begin{theorem}\label{theorem 5.2}
Let $y$ be the solution of $\eqref{5.1}$  and $y^{\delta}$ be the solution of \eqref{5.2} for $\delta\in (0,1)$. For any $\tau\in\mathbb{R}$ and $T>0$. if $\|y^{\delta}_{\tau}-y_{\tau}\|\rightarrow 0$ for $\delta\rightarrow 0$, then we have
\begin{equation}
\|y-y^{\delta}\|_{\infty,[\tau,\tau+T]}+\interleave y-y^{\delta}\interleave_{p-var,[\tau,\tau+T]}+\interleave R^{y^{\delta}}- R^{y}\interleave_{q-var,[\tau,\tau+T]^{2}}\rightarrow 0,
\end{equation}
where $q=\frac{p}{2}$.
\end{theorem}
\begin{proof}
For any $s<t\in[\tau,\tau+T]$, let us transform \eqref{5.1} and \eqref{5.2} into integral form
\begin{equation}
y_{s,t}=\int_{s}^{t}Ay_{r}+f(y_{r})dr+\int_{s}^{t}g(y_{r})d\boldsymbol{\omega}_{r}
\end{equation}
and
\begin{equation}
y^{\delta}_{s,t}=\int_{s}^{t}Ay^{\delta}_{r}+f(y^{\delta}_{r})dr+\int_{s}^{t}g(y^{\delta}_{r})d\boldsymbol{\omega}_{\delta,r}.
\end{equation}
Then we have
\begin{equation}\label{5.18}
\begin{aligned}
\|y_{t}-y^{\delta}_{t}-y_{s}+y^{\delta}_{s}\|&=\|y_{s,t}-y^{\delta}_{s,t}\|\\
&\leq \left\|\int_{s}^{t}A(y_{r}-y_{r}^{\delta})+f(y_{r})-f(y_{r}^{\delta})dr\right\|\\
&~~~~+\left\|\int_{s}^{t}g(y_{r})d\boldsymbol{\omega}_{r}-\int_{s}^{t}g(y_{r}^{\delta})d\boldsymbol{\omega}_{\delta,r}\right\|\\
&\leq \int_{s}^{t}L\|y_{r}-y_{r}^{\delta}\|dr+\left\|\int_{s}^{t}g(y_{r})d\omega^{1}_{r}-\int_{s}^{t}g(y_{r}^{\delta})d\boldsymbol{\omega}_{\delta,r}\right\|\\
&=\int_{s}^{t}L\|y_{r}-y_{r}^{\delta}\|dr+\|Z_{s,t}-Z^\delta_{s,t}\|,
\end{aligned}
\end{equation}
where $Z_{s,t}:=\int_{s}^{t}g(y_{r})d\boldsymbol{\omega}_{r}$ and $Z^{\delta}_{s,t}:=\int_{s}^{t}g(y_{r}^{\delta})d\boldsymbol{\omega}_{\delta,r}$.
By \eqref{5.18}, Theorem \ref{theorem 7.1} in the Appendix, we obtain
\begin{small}
	\begin{align}\label{5.39}
&\interleave \!y\!-\!y^{\delta}\interleave_{p-var,[s,t]}\leq \int_{s}^{t}L\|y_{r}-y_{r}^{\delta}\|dr+\!\!15C_{p}(C_{g}^{2}\interleave\! \boldsymbol{\omega}\!\interleave^{2}_{p-var,[s,t]}\!\vee C_{g}\interleave\! \boldsymbol{\omega}\interleave_{p-var,[s,t]})\nonumber\\
&~~~~~\times\left(\interleave y^{\delta}\!\interleave_{p-var,[s,t]}\!+\!\interleave y\interleave_{p-var,[s,t]}+\interleave R^{y}\interleave_{q-var,[s,t]^{2}}+1\right)\nonumber\\
&~~~~~\times\left(\interleave y-y^{\delta}\interleave_{p-var,[s,t]}+\|y-y^{\delta}\|_{\infty,[s,t]}+\interleave R^{y}-R^{y^{\delta}}\interleave_{q-var,[s,t]^{2}}\right)\nonumber\\
&~~~~~+\left(\interleave y^{\delta}\interleave_{p-var,[s,t]}(\interleave W_{\delta}(\cdot,\omega^{1})\interleave_{p-var,[s,t]}\!+\!\interleave\omega^{1}\interleave_{p-var,[s,t]})\right.\nonumber\\
&~~~~~\left.+\!\interleave R^{y^{\delta}}\interleave_{q-var,[s,t]^{2}}+1\right)\!C_{g}^{2}\vee C_{g}\interleave \omega^{1}\!-\!W_{\delta}(\cdot,\omega^{1})\interleave_{p-var,[s,t]}\nonumber\\
&~~~~~+2C_{g}^{2}C_{p}\left[\interleave y^{\delta} \interleave_{p-var,[s,t]}+1\right]\interleave\omega^{2}-\mathbb{W}_{\delta}(\omega^{1})\interleave_{q-var,[s,t]^{2}}.
\end{align}\end{small}
Since $\|y-y^{\delta}\|_{\infty,[s,t]}\leq \|y_{s}-y^{\delta}_{s}\|+\interleave y-y^{\delta}\interleave_{p-var,[s,t]}$, and by $R^{y}=y_{s,t}-g(y_{s})\omega^{1}_{s,t}, R^{y^{\delta}}=y^{\delta}_{s,t}-g(y^{\delta}_{s})W_{\delta}(\cdot,\omega^{1})_{s,t}$, \eqref{5.21}, we get
 an estimate for $\|y-y^{\delta}\|_{\infty,[s,t]}$ and $\interleave R^y-R^{y^\delta}\interleave_{q-var,[s,t]^2}$. Furthermore, we have
\begin{small}
\begin{align}\label{5.42}
&\interleave \!y\!-\!y^{\delta}\!\interleave_{p-var,[s,t]}\!+\!\| y\!-\!y^{\delta}\|_{\infty,[s,t]}+ \interleave \!R^{y}\!-\!R^{y^{\delta}}\interleave_{q-var,[s,t]^{2}}\leq  \int_{s}^{t}3L\|y_{r}-y_{r}^{\delta}\|dr\nonumber\\
&~~~~+ 46C_{p}(C_{g}^{2}\interleave \boldsymbol{\omega}\interleave^{2}_{p-var,[s,t]}\vee C_{g}\interleave \boldsymbol{\omega}\interleave_{p-var,[s,t]})\left(\interleave y^{\delta}\interleave_{p-var,[s,t]}\right.\nonumber\\
&~~~~\left.+\interleave y\interleave_{p-var,[s,t]}+\interleave R^{y}\interleave_{q-var,[s,t]^{2}}+1\right)\left(\interleave y-y^{\delta}\interleave_{p-var,[s,t]}\right.\nonumber\\
&~~~~\left.+\|y-y^{\delta}\|_{\infty,[s,t]}+\interleave R^{y}-R^{y^{\delta}}\interleave_{q-var,[s,t]^{2}}\right)+\|y_{s}-y^{\delta}_{s}\|\nonumber\\
&~~~~+\left(\interleave y^{\delta}\interleave_{p-var,[s,t]}(\interleave W_{\delta}(\cdot,\omega^{1})\interleave_{p-var,[s,t]}\!+\!\interleave\omega^{1}\interleave_{p-var,[s,t]})\nonumber\right.\\
&~~~~\left.+\interleave R^{y^{\delta}}\interleave_{q-var,[s,t]^{2}}+1\right)4(C_{g}^{2}\vee C_{g})\interleave \omega^{1}\!-\!W_{\delta}(\cdot,\omega^{1})\interleave_{p-var,[s,t]}\nonumber\\
&~~~~+6C_{g}^{2}C_{p}\left[\interleave y^{\delta} \interleave_{p-var,[s,t]}+1\right]\interleave\omega^{2}-\mathbb{W}_{\delta}(\omega^{1})\interleave_{q-var,[s,t]^{2}}.
\end{align}
\end{small}
Based on the estimates of the solution on a finite interval, we consider another sequence of stopping times. For  $\delta\in (0,1]$, we choose
 $$\tilde{\eta}^{\delta}= \frac{1}{92C_{p}C_{g}[\interleave y^{\delta}\interleave_{p-var,[\tau,\tau+T]}+\interleave y\interleave_{p-var,[\tau,\tau+T]}+\interleave R^{y}\interleave_{q-var,[\tau,\tau+T]^{2}}+1]}.$$
Let
$$\tilde{\tau}^{\delta}_{0}=\tau,\tilde{\tau}^{\delta}_{i+1}:=\inf\{t>\tilde{\tau}^\delta_{i};\interleave\boldsymbol{\omega}\interleave_{p-var,[\tilde{\tau}^{\delta}_{i},t]}=\tilde{\eta}^{\delta}\}\wedge(\tau+T)$$
such that $46C_{p}C_{g}\interleave \boldsymbol{\omega}\interleave_{p-var,[\tilde{\tau}_{i},\tilde{\tau}_{i+1}]}[\interleave y^{\delta}\interleave_{p-var,[\tau,\tau+T]}+\interleave y\interleave_{p-var,[\tau,\tau+T]}+\\ \interleave R^{y}\interleave_{q-var,[\tau,\tau+T]^{2}}+1]\leq \frac{1}{2}$ which implies $C_{g}^{2}\interleave\boldsymbol{\omega}\interleave_{p-var,[s,t]}^2\leq C_{g}\interleave\boldsymbol{\omega}\interleave_{p-var,[s,t]}$. Furthermore,  Lemma \ref{theorem 5.1} and Lemma \ref{lemma 5.3} show  that $\interleave y,R^y\interleave_{p-var,[\tau,\tau+T]}$ and $\interleave y^{\delta}\interleave_{p-var,[\tau,\tau+T]}$ can be estimated by the number of stopping times for $\boldsymbol{\omega}$ and  $\boldsymbol{\omega}_{\delta}$,  in addition,  Theorem \ref{theorem 7.2} tells us  that  the sequence of stopping times $\{\tau^\delta_{i}\}$ converging to $\{\tau_i\}$, then the number of stopping times also converges. Thus, the  number  of stopping times for $\boldsymbol{\omega}$ and  $\boldsymbol{\omega}_{\delta}$  are bounded for any $T>0$, then  we have $\inf_{\delta\in(0,1]}\tilde{\eta}^{\delta}>0$.

Thus, for any $t\in [\tilde{\tau}^{\delta}_{i},\tilde{\tau}^{\delta}_{i+1}]$, we have
 \begin{align}
&\interleave y-y^{\delta}\interleave_{p-var,[\tilde{\tau}^{\delta}_{i},t]}+\| y-y^{\delta}\|_{\infty,[\tilde{\tau}^{\delta}_{i},t]}+ \interleave R^{y}-R^{y^{\delta}}\interleave_{q-var,[\tilde{\tau}^{\delta}_{i},t]^{2}}\nonumber\\
&\leq 6L \int_{\tilde{\tau}^{\delta}_{i}}^{t}\interleave y-y^{\delta}\interleave_{p-var,[\tilde{\tau}^{\delta}_{i},r]}+\| y-y^{\delta}\|_{\infty,[\tilde{\tau}^{\delta}_{i},r]}+\interleave R^{y}-R^{y^{\delta}}\interleave_{p-var,[\tilde{\tau}^{\delta}_{i},r]^{2}}dr\nonumber\\
&+2K_{\triangle_{[\tilde{\tau}^{\delta}_{i},\tilde{\tau}^{\delta}_{i+1}]}}(\interleave \omega^{1}-W_{\delta}(\cdot,\omega^{1})\interleave_{p-var,[\tilde{\tau}^{\delta}_{i},\tilde{\tau}^{\delta}_{i+1}]},\interleave \omega^{2}-\mathbb{W}_{\delta}(\omega^{1})\interleave_{q-var,[\tilde{\tau}^{\delta}_{i},\tilde{\tau}^{\delta}_{i+1}]^{2}})\nonumber\\
&+\|y_{\tilde{\tau}^{\delta}_{i}}-y^{\delta}_{\tilde{\tau}^{\delta}_{i}}\|(2+6L(\tilde{\tau}^{\delta}_{i+1}-\tilde{\tau}^{\delta}_{i})),
\end{align}
 where we denote by
\begin{align}
 &K_{\triangle_{[\tilde{\tau}^{\delta}_{i},\tilde{\tau}^{\delta}_{i+1}]}}=K_{\triangle_{[\tilde{\tau}^{\delta}_{i},\tilde{\tau}^{\delta}_{i+1}]}}(\interleave \omega^{1}-W_{\delta}(\cdot,\omega^{1})\interleave_{p-var,[\tilde{\tau}^{\delta}_{i},\tilde{\tau}^{\delta}_{i+1}]},\interleave \omega^{2}-\mathbb{W}_{\delta}(\omega^{1})\interleave_{q-var,[\tilde{\tau}_{i},\tilde{\tau}_{i+1}]})\nonumber\\
 &=\left(\interleave y^{\delta}\interleave_{p-var,[\tilde{\tau}^{\delta}_{i},\tilde{\tau}^{\delta}_{i+1}]}(\interleave W_{\delta}(\cdot,\omega^{1})\interleave_{p-var,[\tilde{\tau}^{\delta}_{i},\tilde{\tau}^{\delta}_{i+1}]}+\interleave\omega^{1}\interleave_{p-var,[\tilde{\tau}^{\delta}_{i},\tilde{\tau}^{\delta}_{i+1}]})\nonumber\right.\\
&\left.+\interleave R^{y^{\delta}}\interleave_{q-var,[\tilde{\tau}^{\delta}_{i},\tilde{\tau}^{\delta}_{i+1}]^{2}}+1\right)4(C_{g}^{2}\!\vee\! C_{g})\interleave\! \omega^{1}\!-\!W_{\delta}(\cdot,\omega^{1})\!\interleave_{p-var,[\tilde{\tau}^{\delta}_{i},\tilde{\tau}^{\delta}_{i+1}]}\nonumber\\
&+6C_{g}^{2}C_{p}\left[\interleave y^{\delta} \interleave_{p-var,[\tilde{\tau}^{\delta}_{i},\tilde{\tau}^{\delta}_{i+1}]}+1\right]\interleave\omega^{2}-\mathbb{W}_{\delta}(\omega^{1})\interleave_{q-var,[\tilde{\tau}^{\delta}_{i},\tilde{\tau}^{\delta}_{i+1}]^{2}}.
\end{align}
Note that Lemma \ref{theorem 5.1}, Theorem \ref{theorem 7.1}, the estimates of the number of stopping times ensure that $\sup_{\delta\in(0,1]}(\interleave y^{\delta}\interleave_{p-var,[\tau,\tau+T]}+\interleave R^{y^{\delta}}\interleave_{q-var,[\tau,\tau+T]^{2}})<\infty$ and the term $K_{\triangle_{[\tilde{\tau}^{\delta}_{i},\tilde{\tau}^{\delta}_{i+1}]}}(\cdot)$ converge to zero as $\delta\rightarrow 0$ on each interval $[\tilde{\tau}^\delta_{i},\tilde{\tau}^\delta_{i+1}]$. Thus, by the continuous Gr\"{o}nwall's inequality, we  obtain
\begin{align}\label{4.16**}
&\interleave y-y^{\delta}\interleave_{p-var,[\tilde{\tau}^{\delta}_{i},\tilde{\tau}^{\delta}_{i+1}]}+\| y-y^{\delta}\|_{\infty,[\tilde{\tau}^{\delta}_{i},\tilde{\tau}^{\delta}_{i+1}]}+ \interleave R^{y}-R^{y^{\delta}}\interleave_{q-var,[\tilde{\tau}^{\delta}_{i},\tilde{\tau}^{\delta}_{i+1}]^{2}}\nonumber\\
&~~~~\leq C_{L,T}\|y_{\tilde{\tau}^{\delta}_{i}}-y_{\tilde{\tau}^{\delta}_{i}}\|e^{6L(\tilde{\tau}^{\delta}_{i+1}-\tilde{\tau}^{\delta}_{i})}+C_{L,T}e^{6L(\tilde{\tau}^{\delta}_{i+1}-\tilde{\tau}^{\delta}_{i})}K_{\triangle_{[\tilde{\tau}^{\delta}_{i},\tilde{\tau}^{\delta}_{i+1}]}},
\end{align}
where $C_{L,T}$ is a constant which only depends on $L$ and $T$. We now are in a position to  consider $\|y-y^{\delta}\|_{\infty,[\tau,\tau+T]}$. Using the above inequality \eqref{4.16**} we get
\begin{small}
\begin{align}
&\|y_{\tilde{\tau}^{\delta}_{i+1}}-y^{\delta}_{\tilde{\tau}^{\delta}_{i+1}}\|\leq \|y-y^{\delta}\|_{\infty,[\tilde{\tau}^{\delta}_{i},\tilde{\tau}^{\delta}_{i+1}]}\nonumber\\
&~~~~\leq C_{L,T}\|y_{\tilde{\tau}^{\delta}_{i}}-y_{\tilde{\tau}^{\delta}_{i}}\|e^{6L(\tilde{\tau}^{\delta}_{i+1}-\tilde{\tau}^{\delta}_{i})}+C_{L,T}e^{6L(\tilde{\tau}^{\delta}_{i+1}-\tilde{\tau}^{\delta}_{i})}K_{\triangle_{[\tilde{\tau}^{\delta}_{i},\tilde{\tau}^{\delta}_{i+1}]}}\nonumber\\
&~~~~\leq C_{L,T}\|y-y^{\delta}\|_{\infty,[\tilde{\tau}^{\delta}_{i-1},\tilde{\tau}^{\delta}_{i}]}e^{6L(\tilde{\tau}^{\delta}_{i+1}-\tilde{\tau}^{\delta}_{i})}+C_{L,T}e^{6L(\tilde{\tau}^{\delta}_{i+1}-\tilde{\tau}^{\delta}_{i})}K_{\triangle_{[\tilde{\tau}^{\delta}_{i},\tilde{\tau}^{\delta}_{i+1}]}}\nonumber\\
&~~~~\leq C_{L,T}^{i+1}\|y_{\tau}-y_{\tau}^{\delta}\|e^{6L(\tilde{\tau}^{\delta}_{i+1}-\tau)}+\sum_{j=0}^{i}C_{L,T}^{j+1}e^{6L(\tilde{\tau}^{\delta}_{i+1}-\tilde{\tau}^{\delta}_{i-j})}K_{\triangle_{[\tilde{\tau}^{\delta}_{i-j},\tilde{\tau}^{\delta}_{i+1-j}]}}.
\end{align}
\end{small}
Then we have
\begin{small}
\begin{align*}
\|y-y^{\delta}\|_{\infty,[\tilde{\tau}^{\delta}_{i},\tilde{\tau}^{\delta}_{i+1}]}\leq C_{L,T}^{\tilde{N}^\delta}\|y_{\tau}-y_{\tau}^{\delta}\|e^{6LT}+\sum_{j=0}^{\tilde{N}^\delta-1}C_{L,T}^{j+1}e^{6L(\tilde{\tau}^{\delta}_{i+1}-\tilde{\tau}^{\delta}_{i-j})}K_{\triangle_{[\tilde{\tau}^{\delta}_{i-j},\tilde{\tau}^{\delta}_{i+1-j}]}},
\end{align*}
\end{small}
where the $\tilde{N}^\delta$ is the number of stopping times $\{\tilde{\tau}^{\delta}_{i}\}$, and by Lemma \ref{Lemma 2.1} $\tilde{N}^\delta\leq (\tilde{\eta}^{\delta})^{-p}\interleave \boldsymbol{\omega}\interleave_{p-var,[\tau,\tau+T]}+1$ . Thus,
$$\|y-y^{\delta}\|_{\infty,[\tau,\tau+T]}\rightarrow 0,\quad \delta\rightarrow 0.$$
For $\interleave y-y^{\delta}\interleave_{p-var,[\tau,\tau+T]}$, by Lemma \ref{Lemma 2.1} we obtain
\begin{align}
\interleave y-y^{\delta}\interleave_{p-var,[\tau,\tau+T]}\leq (\tilde{N}^\delta-1)^{\frac{p-1}{p}}\sum_{i=0}^{\tilde{N}^\delta-1}\interleave y-y^{\delta}\interleave_{p-var,[\tilde{\tau}_{i},\tilde{\tau}_{i+1}]}.
\end{align}
 Thus, we can get
 $$\interleave y-y^{\delta}\interleave_{p-var,[\tau,\tau+T]}\rightarrow 0, \quad \delta\rightarrow 0.$$
  Similarly, we also have
   $$\interleave R^{y}-R^{y^{\delta}}\interleave_{p-var,[\tau,\tau+T]}\rightarrow 0, \quad \delta\rightarrow 0.$$
The proof is complete.
\end{proof}
\subsection{Random dynamical systems}
In this subsection, we consider  random dynamical systems which are generated by \eqref{5.1} and \eqref{5.2}. For the  theory of random dynamical systems, we refer to \cite{MR1723992}.  We first to construct an ergodic metric dynamical system.

 For all $\frac{1}{3}<\beta<H<\frac{1}{2}$ and $\omega\in  \Omega$ in remark \ref{remark 2.1}. According to Theorem \ref{theorem 4.6}, there exists a $\theta$-invariant set $\Omega^{\prime}$ such that any $\omega\in\Omega^{\prime}$, the fractional Brownian motion has a canonical lift, namely, the fractional Brownian  rough path, which  we considered in Section \ref{section-intersection}, can be treated as the limit of the canonical lift of smooth path $W_{\delta}(\cdot,\omega^{1})$ and its second order process $\theta_{\tau}\mathbb{W}_{\delta}(\omega^{1})$ is a canonical lift of $\theta_{\tau}W_{\delta}(\cdot,\omega^{1}),\tau\in R,\omega^{1}\in \Omega^{\prime}$.  Thus we restrict the ergodic metric dynamical system given by $(\Omega,\mathcal{F},\mathbb{P}_{H},\theta)$ in Remark \ref{remark 2.1} to $(\Omega^{\prime},\mathcal{F}^{\prime},\mathbb{P}^{\prime}_{H},\theta^{\prime})$, where $\mathcal{F}^{\prime}:=\mathcal{F}\bigcap \Omega^{\prime}$, $\mathbb{P}^{\prime}_{H}$ is  the restriction of $\mathbb{P}_H$ over $\mathcal{F}^\prime$, $\theta^\prime$ denotes the restriction of $\theta$ to $R \times \Omega^ \prime$. Then the metric dynamical system $(\Omega^{\prime},\mathcal{F}^{\prime},\mathbb{P}^{\prime}_{H},\theta^{\prime})$ is  ergodic.
\begin{remark}
As mentioned in the previous sections, we do not need to regard geometric fractional Brownian rough path as a new stochastic process\cite{MR3624539}, in fact the second order process is generated by the path of a fractional Brownian motion, then $\sigma$-algebra $\mathcal{F}^{\prime}$ should be generated by the path not fractional Brownian rough path.
\end{remark}
\begin{theorem}\label{theorem 5.3}
The rough differential equation \eqref{5.1} generates a random dynamical system $\varphi:R^{+}\times \Omega^{\prime}\times R^{m}\mapsto R^{m}$ given by $\varphi(t,\omega,\xi)=y_{t}$ over $(\Omega^{\prime},\mathcal{F}^{\prime},\mathbb{P}^{\prime}_{H},\theta^{\prime})$ and $t\in[0,T]$.
\end{theorem}
\begin{proof}
The measurability of $\varphi$ can be obtained for the continuity with respect to variable $t,\omega,\xi$. Indeed,  the solution continuously depend on  $\omega$ and $\xi$ shall lead to the measurability of $\varphi$  with respect to $\mathcal{F}^\prime\otimes \mathcal{B}(R^{m})$.   Since $\varphi$ is continuous with respect to $t$, we obtain by Lemma 3 in \cite{CastaingValadier} the jointly measurability, i.e. ($\mathcal{B}{R^+} \otimes \mathcal{F}^\prime\otimes \mathcal{B}(R^{m}),\mathcal{B}(R^{m}))$.
The proof of continuity for variable $\omega,\xi$ is similar to Theorem \ref{theorem 5.2}, $y_{t}\in C^{p-var}([0,T],R^{m})\subset C([0,T];R^{m})$. It is trivial $\varphi(0,\omega,\xi)= \xi$.
Then we only need to check the cocycle property
\begin{small}
\begin{align*}
\varphi(t+\tau,\omega,\xi)&=\xi+\int_{0}^{t+\tau}Ay_{r}+f(y_{r})dr+\int_{0}^{t+\tau}g(y_{r})d\boldsymbol{\omega}_{r}\\
&=\xi+\int_{0}^{t}Ay_{r}+f(y_{r})dr+\int_{0}^{t}g(y_{r})d\boldsymbol{\omega}_{r}\\
&~~~~+\int_{t}^{t+\tau}Ay_{r}+f(y_{r})dr+\int_{t}^{t+\tau}g(y_{r})d\boldsymbol{\omega}_{r}\\
&=y_{t}+\int_{0}^{\tau}Ay_{r+t}+f(y_{r+t})dr+\int_{0}^{\tau}g(y_{r+t})d\theta^\prime_{t}\boldsymbol{\omega}_{r}\\
&=\varphi(\tau,\theta^\prime_{t}\omega,\varphi(t,\omega,\xi)),
\end{align*}
\end{small}
where the property $\int_{t}^{t+\tau}g(y_r)d\boldsymbol{\omega}_{r}=\int_{0}^{\tau}g(y_{r+t})d\theta^\prime_{t}\boldsymbol{\omega}_{r}$ can be directly obtained  from the Sewing lemma or (5.2) in \cite{MR4266219}. The same argument  gives us  the additivity  of  rough integral.  Thus, we complete the proof.
\end{proof}
Similar to Theorem \ref{theorem 5.3}, $y_{t}^{\delta}$ in $R^{m}$ can generate a random dynamical system $\varphi^{\delta}$. In addition, Theorem  \ref{theorem 5.2} implies  $\varphi^{\delta}\rightarrow \varphi$ as $\delta\rightarrow 0$.
\appendix
\section{Results on the lift of continuous Gaussian rough paths}\label{Appendix A}
\setcounter{equation}{0}

\renewcommand{\theequation}{B.\arabic{equation}}
In this section, we collect some results that we can use to lift a continuous Gaussian process to a rough path \cite[Chapter 10]{MR4174393}.

For a d-dimensional Gaussian process $X_t$, we need to define the following integral
$$\mathbb{X}^{i,j}_{s,t}=\int_{s}^{t}X^{i}_{s,r}dX^{j}_{r}.$$
Let $\mathcal{P}(s,t)$ be a partition of the interval $[s,t]$ and $\left|\mathcal{P}\right|$ be the maximum length of the partition intervals, and set

$$\int_{\mathcal{P}}X^{i}_{s,r}dX^{j}_{r}:=\sum_{[u,v]\in\mathcal{P}}X^{i}_{ s,u}X^{j}_{u,v}.$$
Under the assumption that $X^{i},X^{j},i\neq j$ are independent, we define
\begin{small}
$$
\int_{\mathcal{P}\times \mathcal{P}^{\prime}}R_{X^{i}}dR_{X^{j}}:=\mathbf{E}\left\{\int_{\mathcal{P}}X^{i}_{r,s}dX^{j}_{r}\int_{\mathcal{P^{\prime}}}X^{i}_{r,s}dX^{j}_{r} \right\}= $$
$$\sum_{\substack{
[u,v]\in\mathcal{P}\\
[u^{\prime},v^{\prime}]\in\mathcal{P}^{\prime}}
}R_{X^{i}}\left(\begin{array}{cc}
s & u \\
s & u^{\prime}
\end{array}\right)R_{X^{j}}\left(\begin{array}{cc}
u & v \\
u^{\prime} & v^{\prime}
\end{array}\right).
$$
\end{small}
Furthermore, if $R_{X^{i}}$ and $R_{X^{j}}$ have a finite $\rho$-variation, according to the Towghi-Young maximal inequality\cite{MR1906391}.  We have

$$\sup_{\substack{
\mathcal{P}\subset[s,t]\\
\mathcal{P}^{\prime}\subset[s,t]}
}\left|\int_{P\times P^{\prime}}R_{X^{i}}dR_{X^{j}}\right|\leq C(\theta)\|R_{X^{i}}\|_{\rho-var;[s,t]}\|R_{X^{j}}\|_{\rho-var;[s,t]},
$$
where $\theta=\frac{2}{\rho}>1$. 
\begin{lemma}[\cite{MR4174393}, Proposition 10.3]\label{lemma 3.1}
Let $X_t$ be a d-dimensional, continuous, centered Gaussian processes  with covariance $R_{X^{i}}$ and $R_{X^{j}}$ for $1\leq i,j\leq d $ and have a $2>\rho$-variation. Then
\begin{equation}\nonumber
\begin{aligned}
\lim_{\epsilon\rightarrow 0}\sup_{\substack{
\mathcal{P}\subset[s,t]\\
\mathcal{P}^{\prime}\subset[s,t]\\
|\mathcal{P}|\vee |\mathcal{P}^{\prime}|<\epsilon}
}\left|\int_{\mathcal{P}}X_{s,r}^{i}dX_{r}^{j}-\int_{\mathcal{P}^{\prime}}X_{s,r}^{i}dX_{r}^{j}\right|_{L^{2}}=0.
\end{aligned}
\end{equation}
Thus, $\int_{s}^{t}X_{s,r}^{i}dX_{r}^{j}$ exists as the $L^{2}$ limit of the $\int_{\mathcal{P}}X_{r,s}^{i}dX_{r}^{j}$ as $\left|\mathcal{P}\right|\rightarrow 0$ and

$$\mathbb{E}\left[\left(\int_{s}^{t}X_{s,r}^{i}dX_{r}^{j}\right)^{2}\right]\leq C\|R_{X^{i}}\|_{\rho-var;[s,t]^{2}}\|R_{X^{j}}\|_{\rho-var;[s,t]^{2}},$$
where constant $C$ depends on $\rho$.
\end{lemma}
\begin{lemma}[\cite{MR4174393},Theorem 10.9]\label{lemma 3.2}
Let $X_t$ be a real-valued Gaussian process with stationary increments and $\sigma^{2}_{X}(u)=\mathbb{E}(X_{t+u}-X_{t})^{2}$ be concave and non-decreasing on $[0, h]$ for some $h>0$ and $t\in R$. Further, assume that there exist $\rho \geq 1$ and $L>0$ such that for all $u \in[0, h]$
$$
\left|\sigma_{X}^{2}(u)\right| \leq L u^{\frac{1}{\rho}}.
$$
Then, we have
$$
\left\|R_{X}\right\|_{\rho-var,[s, t]^{2}} \leq M(t-s)^{\frac{1}{\rho}}
$$
for all $[s, t]$ with $|t-s| \leq h$ and $M=M(\rho, L)$.
\end{lemma}
\begin{theorem}[\cite{MR4174393},Theorem 10.4]\label{theorem 3.1}
Let $\left(X_{t}\right)_{t \in[0, T]}$ be a d-dimensional, continuous Gaussian process with independent components and covariance function $R$ such that there exists a $\rho \in[1,2)$ and $M>0$ such that for every $i \in\{1, \ldots, d\}$,

$$
\left\|R_{X^{i}}\right\|_{\rho-var,[s, t]^{2}} \leq M(t-s)^{\frac{1}{\rho}}, \quad \text { for } 0 \leq s \leq t \leq T.
$$
Then, we define for $1 \leq i<j \leq d$

$$
\begin{aligned}
\mathbb{X}_{s, t}^{i, j} &=\lim _{|\mathcal{P}| \rightarrow 0} \int_{\mathcal{P}}X^{i}_{s, r} dX_{r}^{j} \quad \text { in } L^{2} \text { sense}, \\
\mathbb{X}_{s, t}^{i, i} &=\frac{1}{2} \left(X^{i}_{s, t}\right)^{2}  \quad \text{and} \quad
\mathbb{X}_{s, t}^{ j, i} =-\mathbb{X}_{s, t}^{ i, j}+X^{i}_{s, t} X^{j}_{s, t}.
\end{aligned}
$$
Further, the following properties hold:\\
i) For every $q \geq 1$ there exists $C=C(q, \rho, d, T)$ such that

$$
\mathbb{E}\left(\left|X_{s, t}\right|^{2 q}+\left|\mathbb{X}_{s, t}^{(2)}\right|^{q}\right) \leq C M^{q}(t-s)^{\frac{q}{\rho}}.
$$
ii) There exists a continuous modification of $\mathbb{X}$ (denoted by the same letter from here on.) Further, for any $\alpha<\frac{1}{2 \rho}$ and $q \geq 1$ there exists $C=C(q, \rho, d, \alpha)$ such that

$$
\mathbb{E}\left(\|X\|_{\alpha}^{2 q}+\left\|\mathbb{X}\right\|_{2 \alpha}^{q}\right) \leq C M^{q}.
$$
iii) For any $\alpha<\frac{1}{2 \rho}, \left(X, \mathbb{X}\right)$ fulfills the Chen equation and \eqref{2.8} with probability one. In particular, for $\rho \in\left[1, \frac{3}{2}\right)$ and any $\alpha \in\left(\frac{1}{3}, \frac{1}{2 \rho}\right)$ we have $\left(X, \mathbb{X}\right) \in \mathscr{C}_{g}^{\alpha}$ a.s.
\end{theorem}
\section{Appendix B}\label{appendix B}
\setcounter{equation}{0}

\renewcommand{\theequation}{B.\arabic{equation}}
\begin{theorem}\label{theorem 7.1}
Let $Z_{s,t}=\int_{s}^{t}g(y_{r})d\boldsymbol{\omega}_{r}$ be the rough integral for \eqref{5.1}, $Z^{\delta}_{s,t}=\int_{s}^{t}g(y^\delta_{r})d\boldsymbol{\omega}_{\delta,r}$ which emerges in \eqref{5.2} and can be understood a rough integral.  Then the following estimate holds
\begin{small}
\begin{align*}
&\| Z_{s,t}-Z^{\delta}_{s,t}\|\leq 15C_{p}(C_{g}^{2}\interleave \boldsymbol{\omega}\interleave^{2}_{p-var,[s,t]}\vee C_{g}\interleave \boldsymbol{\omega}\interleave_{p-var,[s,t]})\nonumber\\
&~~~~\times\left(\interleave y^{\delta}\interleave_{p-var,[s,t]}+\interleave y\interleave_{p\!-\!var,[s,t]}\!+\!\interleave R^{y}\interleave_{q\!-\!var,[s,t]^{2}}\!+\!1\right)\nonumber\\
&~~~~\times\left(\interleave y\!-\!y^{\delta}\interleave_{p\!-\!var,[s,t]}\!+\!\|y\!-\!y^{\delta}\|_{\infty,[s,t]}\!+\!\interleave R^{y}\!-\!R^{y^{\delta}}\interleave_{q\!-\!var,[s,t]^{2}}\right)\nonumber\\
&~~~~+\left(\interleave y^{\delta}\interleave_{p-var,[s,t]}(\interleave W_{\delta}(\cdot,\omega^{1})\interleave_{p-var,[s,t]}\!+\!\interleave\omega^{1}\interleave_{p-var,[s,t]})\nonumber\right.\\
&~~~~\left.+\interleave R^{y^{\delta}}\interleave_{q-var,[s,t]^{2}}+1\right)C_{g}^{2}\vee C_{g}\interleave \omega^{1}\!-\!W_{\delta}(\cdot,\omega^{1})\interleave_{p-var,[s,t]}\nonumber\\
&~~~~+2C_{g}^{2}C_{p}\left[\interleave y^{\delta} \interleave_{p-var,[s,t]}+1\right]\interleave\omega^{2}-\mathbb{W}_{\delta}(\omega^{1})\interleave_{q-var,[s,t]^{2}}.
\end{align*}
\end{small}
\end{theorem}
\begin{proof}
We first estimate $\|Z_{s,t}-Z^\delta_{s,t}\|$, applying  Definition  \ref{def 5.1} we have
\begin{align}\label{5.19}
\|Z_{s,t}-Z^{\delta}_{s,t}\|&=\|g(y_{s})\omega^{1}_{s,t}-g(y_{s}^{\delta})W_{\delta}(\cdot,\omega^{1})_{s,t}+R^{Z}_{s,t}-R^{Z^\delta}_{s,t}\|\nonumber\\
&\leq \|g(y_{s})\omega^{1}_{s,t}-g(y_{s}^{\delta})W_{\delta}(\cdot,\omega^{1})_{s,t}\|+ \|R^{Z}_{s,t}-R^{Z^\delta}_{s,t}\|\nonumber\\
&=:A_{1}+A_{2}.
\end{align}
For $A_{1}$ we have
\begin{align}\label{5.20}
A_{1}&\leq \|g(y_{s})\omega^{1}_{s,t}-g(y^{\delta}_{s})\omega_{s,t}^{1}\|+\|g(y^{\delta}_{s})(\omega_{s,t}^{1}-W_{\delta}(\cdot,\omega^{1})_{s,t})\|\nonumber\\
&\leq C_{g}\|y-y^{\delta}\|_{\infty,[s,t]}\interleave \omega^{1} \interleave_{p-var,[s,t]}+C_{g}\interleave\omega^{1}- W_{\delta}(\cdot,\omega^{1})\interleave_{p-var,[s,t]}.
\end{align}
For $A_{2}$, let
\begin{equation}
\Xi_{s,t}:=g(y_{s})\omega^{1}_{s,t}+Dg(y_{s})g(y_{s})\omega^{2}_{s,t},
\end{equation}
\begin{equation}
\Xi^\delta_{s,t}:=g(y^{\delta}_{s})W_{\delta}(\cdot,\omega^{1})_{s,t}+Dg(y^{\delta}_{s})g(y^{\delta}_{s})\mathbb{W}_{\delta}(\omega^{1})_{s,t},
\end{equation}
\begin{equation}
\triangle_{s,t}=\Xi_{s,t}-\Xi^\delta_{s,t},
\end{equation}
then  we have that
\begin{align}\label{5.21}
A_{2}&=\|R^{Z}_{s,t}-R^{Z^{\delta}}_{s,t}\|\nonumber\\
&\leq\|\mathcal{I}(\triangle)_{s,t}-\triangle_{s,t}\|+\|Dg(y_{s})g(y_{s})\omega^{2}_{s,t}-Dg(y^{\delta}_{s})g(y^{\delta}_{s})\mathbb{W}_{\delta}(\omega^{1})_{s,t}\|\nonumber\\
&=:B_{1}+B_{2},
\end{align}
where the mapping $\mathcal{I}: \Xi_{s,t}\mapsto Z_{s,t}$  and   $\mathcal{I}: \Xi^{\delta}_{s,t}\mapsto Z^{\delta}_{s,t}$.  For $B_{1}$, according to the linearity of the mapping $\mathcal{I}$, we have another version of  $\eqref{5.10}$ for $\triangle$. Thus we have
\begin{align}
B_{1}\leq C_{p}\interleave \delta\triangle \interleave_{\frac{p}{3}-var},
\end{align}
where
\begin{small}
\begin{align}
(\delta\triangle)_{s,u,t}&=R^{g(y)}_{s,u}\omega^{1}_{u,t}-R_{s,u}^{g(y^{\delta})}W_{\delta}(\cdot,\omega^{1})_{u,t}\nonumber\\
&~~~~+(Dg(y_{\cdot})g(y_{\cdot}))_{s,u}\omega^{2}_{u,t}-(Dg(y^{\delta}_{\cdot})g(y^{\delta}_{\cdot}))_{s,u}\mathbb{W}_{\delta}(\omega^{1})_{u,t}\nonumber\\
&=:B_{1,1}+B_{1,2}.
\end{align}
\end{small}
For $B_{1,1}$, we have
\begin{align}\label{5.24}
\|B_{1,1}\|&=\|R^{g(y)}_{s,u}\omega^{1}_{u,t}-R_{s,u}^{g(y^{\delta})}W_{\delta}(\cdot,\omega^{1})_{u,t}\|\nonumber\\
&\leq \|R^{g(y)}_{s,u}-R^{g(y^{\delta})}_{s,u}\|\|\omega^{1}_{u,t}\|+\|R^{g(y^{\delta})}_{s,u}\|\|\omega^{1}_{u,t}-W_{\delta}(\cdot,\omega^{1})_{u,t} \|,
\end{align}
where
\begin{align}
R^{g(y)}_{s,u}&=g(y_{\cdot})_{s,u}-Dg(y_{s})g(y_{s})\omega^{1}_{s,u}\nonumber\\
&=\int_{0}^{1}Dg(y_{s}+ry_{s,u})R^{y}_{s,u}dr+\int_{0}^{1}\left[Dg(y_{s}+ry_{s,u})-Dg(y_{s})\right]g(y_{s})\omega^{1}_{s,u}dr\nonumber
\end{align}
and
\begin{align}
&R^{g(y^{\delta})}_{s,u}=g(y^{\delta})_{s,u}-Dg(y^{\delta}_{s})g(y^{\delta}_{s})W_{\delta}(\cdot,\omega^{1})_{s,u}\nonumber\\
&~~~=\int_{0}^{1}Dg(y^{\delta}_{s}+ry^{\delta}_{s,u})R^{y^{\delta}}_{s,u}dr+\int_{0}^{1}\left[Dg(y^{\delta}_{s}+ry^{\delta}_{s,u})-Dg(y^{\delta}_{s})\right]g(y^{\delta}_{s})W_{\delta}(\cdot,\omega^{1})_{s,u}dr,\nonumber
\end{align}
the above identities hold   by  $y_{s,u}=y_s\omega_{s,u}+R^{y}_{s,u}$, $y^\delta_{s,u}=y^\delta_s\omega_{s,u}+R^{y^\delta}_{s,u}$,  and  Remark \ref{Rem4.2} shows the term $A+f$ should be contained in the remainder term $R^{y}$.  Hence, we have
\begin{align}\label{5.25}
\|R^{g(y)}_{s,u}\!-\!R^{g(y^{\delta})}_{s,u}\|&\!\leq\! \left\|\!\int_{0}^{1}\!Dg(y_{s}\!+\!ry_{s,u})R^{y}_{s,u}dr\!-\!\int_{0}^{1}\!Dg(y^{\delta}_{s}\!+\!ry^{\delta}_{s,u})R^{y^{\delta}}_{s,u}dr\right\|\nonumber\\
&~~+\left\|\int_{0}^{1}\left[Dg(y_{s}+ry_{s,u})-Dg(y_{s})\right]g(y_{s})\omega^{1}_{s,u}dr\right.\nonumber\\
&~~\left.-\int_{0}^{1}\!\left[Dg(y^{\delta}_{s}\!+\!ry^{\delta}_{s,u})-Dg(y^{\delta}_{s})\right]g(y^{\delta}_{s})W_{\delta}(\cdot,\omega^{1})_{s,u}dr\right\|\nonumber\\
&=:R_{1,g}+R_{2,g}.
\end{align}
For $R_{1,g}$, we  obtain
\begin{align}\label{5.26}
R_{1,g}&\leq \left\|\int_{0}^{1}\left(Dg(y_{s}+ry_{s,u})-Dg(y_{s}^{\delta}+ry_{s,u}^{\delta})\right)R^{y}_{s,u}dr\right\|\nonumber\\
&~~+\left\|\int_{0}^{1}Dg(y_{s}^{\delta}+ry_{s,u}^{\delta})(R^{y}_{s,u}-R^{y^{\delta}}_{s,u})dr\right\|\nonumber\\
&\leq 3C_{g}\|y-y^{\delta}\|_{\infty,[s,t]}\|R^{y}_{s,u}\|+C_{g}\|R^{y}_{s,u}-R^{y^{\delta}}_{s,u}\|\nonumber\\
&\leq 3C_{g}\|y-y^{\delta}\|_{\infty,[s,t]}\interleave R^{y}\interleave_{q-var,[s,t]^{2}}+C_{g}\interleave R^{y}-R^{y^{\delta}}\interleave_{q-var,[s,t]^{2}}.
\end{align}
For $R_{2,g}$, we have
\begin{align}\label{5.27}
R_{2,g}&\leq\left\|\int_{0}^{1}\!\left(\!Dg(y_{s}\!+\!ry_{s,u})\!-\!Dg(y_{s})\!-\!Dg(y^{\delta}_{s}\!+\!ry^{\delta}_{s,u})\!+\!Dg(y_{s}^{\delta})\right)g(y_{s})\omega^{1}_{s,u}dr\right\|\nonumber\\
&+\left\|\int_{0}^{1}\left[Dg(y_{s}^{\delta}+ry_{s,u}^{\delta})-Dg(y^{\delta}_{s})\right]\left[g(y_{s})\omega_{s,u}^{1}-g(y^{\delta}_{s})W_{\delta}(\cdot,\omega^{1})_{s,u}\right]dr\right\|\nonumber\\
&\leq \left\|\int_{0}^{1}\left(\!Dg(y_{s}\!+\!ry_{s,u})\!-\!Dg(y_{s})\!-\!Dg(y^{\delta}_{s}\!+\!ry^{\delta}_{s,u})\!+\!Dg(y_{s}^{\delta})\right)g(y_{s})\!\omega^{1}_{s,u}\!dr\right\|\nonumber\\
&+\left\|\int_{0}^{1}\left[Dg(y_{s}^{\delta}+ry_{s,u}^{\delta})-Dg(y^{\delta}_{s})\right]\left[(g(y_{s})-g(y^{\delta}_{s}))\omega^{1}_{s,u}\right]dr\right\|\nonumber\\
&+\left\|\int_{0}^{1}\left[Dg(y_{s}^{\delta}+ry_{s,u}^{\delta})-Dg(y^{\delta}_{s})\right]\left[g(y^{\delta}_{s})\left(\omega^{1}_{s,u}-W_{\delta}(\cdot,\omega^{1})_{s,u}\right)\right]dr\right\|\nonumber\\
&=:\mathcal{R}_{1}+\mathcal{R}_{2}+\mathcal{R}_{3}.
\end{align}
 We can estimate $\mathcal{R}_{1}$ as follows
\begin{align}\label{5.28}
\mathcal{R}_{1}&\leq C_{g}^{2}\|y_{s,u}-y_{s,u}^{\delta}\| \|\omega^{1}_{s,u}\|\nonumber\\
&~~+C_{g}^{2}\left[\|y_{s}-y_{s}^{\delta}\|+\|y_{s,u}-y_{s,u}^{\delta}\|\right]\left[\|y_{s,u}\|+\|y_{s,u}^{\delta}\|\right]\|\omega^{1}_{s,u}\|\nonumber\\
&\leq C_{g}^{2}\interleave y-y^{\delta}\interleave_{p-var,[s,t]}\|\omega^{1}\|_{p-var,[s,t]}+3C_{g}^{2}\|y-y^{\delta}\|_{\infty,[s,t]}\nonumber\\
&~~\times\left[\interleave y^{\delta}\interleave_{p-var,[s,t]}+\interleave y\interleave_{p-var,[s,t]}\right]\interleave\omega^{1}\interleave_{p-var,[s,t]},
\end{align}
where we use the following inequality to derive the first inequality:
\begin{align}\label{5.29}
\|h(u_{1})-h(v_{1})-h(u_{2})+h(v_{2})\|&\leq C_{h}\|u_{1}-v_{1}-u_{2}+v_{2}\|\nonumber\\
&+C_{h}\|u_{1}-u_{2}\|\left(\|u_{1}-v_{1}\|+\|u_{2}-v_{2}\|\right),
\end{align}
where $h$ is differentiable, see \cite[Lemma 7.1]{MR1893308}.
For $\mathcal{R}_{2}$ and $\mathcal{R}_{3}$, similar to \eqref{5.28}, we have
\begin{align}
\mathcal{R}_{2}&\leq C_{g}^{2}\|y_{s,u}^{\delta}\|\|y_{s}-y_{s}^{\delta}\|\|\omega^{1}_{s,u}\|\nonumber\\
&\leq C_{g}^{2}\interleave y^{\delta}\interleave_{p-var,[s,t]}\|y-y^{\delta}\|_{\infty,[s,t]}\interleave\omega^{1}\interleave_{p-var,[s,t]},\label{5.30}\\
\mathcal{R}_{3}&\leq C_{g}^{2}\|y_{s,u}^{\delta}\|\|\omega^{1}_{s,u}-W_{\delta}(\cdot,\omega^{1})_{s,u}\|\nonumber\\
&\leq C_{g}^{2}\interleave y^{\delta}\interleave_{p-var,[s,t]}\interleave\omega^{1}-W_{\delta}(\cdot,\omega^{1})\interleave_{p-var,[s,t]}.\label{5.31}
\end{align}
Together with \eqref{5.25}-\eqref{5.31}, we obtain the estimate of the first term of \eqref{5.24}:
\begin{small}
\begin{align}\label{5.32}
\|R^{g(y)}_{s,u}&-R^{g(y^{\delta})}_{s,u}\|\|\omega^{1}_{u,t}\|\leq \left[C_{g}^{2}\interleave y-y^{\delta}\interleave_{p-var,[s,t]}\right.\nonumber\\
&~~~~\left.+3C_{g}^{2}\|y-y^{\delta}\|_{\infty,[s,t]}\left(\interleave y^{\delta}\interleave_{p-var,[s,t]}+\interleave y\interleave_{p-var,[s,t]}\right)\right.\nonumber\\
&~~~~\left.+C_{g}^{2}\|y-y^{\delta}\|_{\infty,[s,t]}\interleave y^{\delta}\interleave_{p-var,[s,t]} \right]\interleave \omega^{1}\interleave^{2}_{p-var,[s,t]}\nonumber\\
&~~~~+\left[C_{g}^{2}\interleave \omega^{1}-W_{\delta}(\cdot,\omega^{1})\interleave_{p-var,[s,t]}\interleave y^{\delta}\interleave_{p-var,[s,t]}\right.\nonumber\\
&~~~~\left.+3C_{g}\|y-y^{\delta}\|_{\infty,[s,t]}\interleave R^{y}\interleave_{q-var,[s,t]^{2}}\right.\nonumber\\
&~~~~\left.+C_{g}\interleave R^{y}-R^{y^{\delta}}\interleave_{q-var,[s,t]^{2}}\right]\interleave\omega^{1}\interleave_{p-var,[s,t]}\nonumber\\
&~~\leq 8C_{g}^{2}\interleave \boldsymbol{\omega}\interleave^{2}_{p-var,[s,t]}\vee 8C_{g}\interleave \boldsymbol{\omega}\interleave_{p-var,[s,t]}\left(\interleave y^{\delta}\interleave_{p-var,[s,t]}\right.\nonumber\\
&~~~~\left.+\interleave y\interleave_{p-var,[s,t]}+\interleave R^{y}\interleave_{q-var,[s,t]^{2}}+1\right)\left(\interleave y-y^{\delta}\interleave_{p-var,[s,t]}\right.\nonumber\\
&~~~~\left.+\|y-y^{\delta}\|_{\infty,[s,t]}+\interleave R^{y}-R^{y^{\delta}}\interleave_{q-var,[s,t]^{2}}\right)\nonumber\\
&~~~~+C_{g}^{2}\interleave \omega^{1}\!-\!W_{\delta}(\cdot,\omega^{1})\interleave_{p\!-\!var,[s,t]}\interleave y^{\delta}\interleave_{p\!-\!var,[s,t]}\interleave \omega^{1}\interleave_{p\!-\!var,[s,t]}.
\end{align}
\end{small}
For the second term of  \eqref{5.24}, we have
\begin{align}\label{5.33}
&\|R^{g(y^{\delta})}_{s,u}\|\|\omega^{1}_{u,t}-W_{\delta}(\cdot,\omega^{1})_{u,t}\|\leq\interleave\omega^{1}-W_{\delta}(\cdot,\omega^{1})\interleave_{p-var,[s,t]}\nonumber\\&~~\times\biggr( \!C_{g}\!\interleave R^{y^{\delta}}\interleave_{q-var,[s,t]^{2}}+C_{g}^{2}\interleave y^{\delta}\interleave_{p-var,[s,t]}\interleave  \interleave W_{\delta}(\cdot,\omega^{1})\interleave_{p-var,[s,t]}\biggr).
\end{align}
Hence, by \eqref{5.32},\eqref{5.33} we obtain
\begin{align}\label{5.34}
\|B_{1,1}\|&\leq 8C_{g}^{2}\interleave \boldsymbol{\omega}\interleave^{2}_{p-var,[s,t]}\vee 8C_{g}\interleave \boldsymbol{\omega}\interleave_{p-var,[s,t]}\left(\interleave y^{\delta}\interleave_{p-var,[s,t]}\right.\nonumber\\
&~~\left.+\interleave y\interleave_{p-var,[s,t]}+\interleave R^{y}\interleave_{q-var,[s,t]^{2}}+1\right)\left(\interleave y-y^{\delta}\interleave_{p-var,[s,t]}\right.\nonumber\\
&~~\left.+\|y-y^{\delta}\|_{\infty,[s,t]}+\interleave R^{y}-R^{y^{\delta}}\interleave_{q-var,[s,t]^{2}}\right)\nonumber\\
&~~+(C_{g}^{2}\vee C_{g})\interleave \omega^{1}\!-\!W_{\delta}(\cdot,\omega^{1})\interleave_{p-var,[s,t]}\left((\interleave W_{\delta}(\cdot,\omega^{1})\interleave_{p-var,[s,t]}\right.\nonumber\\
&~~\left.+\interleave\omega^{1} \interleave_{p-var,[s,t]})\interleave y^{\delta}\interleave_{p-var,[s,t]}+\interleave R^{y^{\delta}}\interleave_{q-var,[s,t]^{2}}\right).
\end{align}
For $B_{1,2}$, using \eqref{5.29} we have
\begin{align}\label{3.35}
&\|B_{1,2}\|=\|(g^{\prime}(y_{\cdot}))_{s,u}\omega^{2}_{u,t}-(g^{\prime}(y^{\delta}_{\cdot}))_{s,u}\mathbb{W}_{\delta}(\omega^{1})_{u,t}\|\nonumber\\
&~\leq 2 C_{g}^{2}\|y^{\delta}_{s,u}\|\|\omega^{2}_{u,t}-\mathbb{W}_{\delta}(\omega^{1})_{u,t}\| \nonumber\\
&~~~~+\left[ 2C_{g}^{2}\|y_{s,u}-y^{\delta}_{s,u}\|+4C_{g}^{2}\|y_{u}-y_{u}^{\delta}\|\left(\|y_{s,u}\|+\|y_{s,u}^{\delta}\|\right)\right]\|\omega^{2}_{u,t}\|\nonumber\\
&~\leq \!2C_{g}^{2}\!\interleave\! y^{\delta}\!\interleave_{p-var,[s,t]}\interleave\omega^{2}\!-\!\mathbb{W}_{\delta}(\omega^{1})\interleave_{q-var,[s,t]^{2}}\!+4C_{g}^{2}\left[\interleave y-y^{\delta}\interleave_{p-var,[s,t]}\right.\nonumber\\ &~~~~\left.+\|y-y^{\delta}\|_{\infty,[s,t]}\left[\interleave y\interleave_{p-var,[s,t]}+\interleave y^{\delta}\interleave_{p-var,[s,t]}\right] \right]\interleave\omega^{2}\interleave_{q-var,[s,t]^{2}}\nonumber\\
&~\leq 4C_{g}^{2}\left[\interleave y-y^{\delta}\interleave_{p-var,[s,t]}+\|y-y^{\delta}\|_{\infty,[s,t]}\right]\nonumber\\
&~~~~\times\left[\interleave y \interleave_{p-var,[s,t]}+\interleave y^{\delta} \interleave_{p-var,[s,t]}+1\right]\interleave \omega^{2}\interleave_{q-var,[s,t]^{2}}\nonumber\\
&~~~~+2C_{g}^{2}\interleave y^{\delta} \interleave_{p-var,[s,t]}\interleave\omega^{2}-\mathbb{W}_{\delta}(\omega^{1})\interleave_{q-var,[s,t]^{2}}.
\end{align}
Combining  \eqref{5.34} and \eqref{3.35}, we get
\begin{small}
\begin{align}\label{5.36}
&B_{1}\leq 12C_{p}(C_{g}^{2}\interleave \boldsymbol{\omega}\interleave^{2}_{p-var,[s,t]}\vee C_{g}\interleave \boldsymbol{\omega}\interleave_{p-var,[s,t]})\nonumber\\
&~~~~~~\times\left(\interleave y^{\delta}\interleave_{p-var,[s,t]}+\interleave y\interleave_{p-var,[s,t]}+\interleave R^{y}\interleave_{q-var,[s,t]^{2}}+1\right)\nonumber\\
&~~~~~~\times\left(\interleave y-y^{\delta}\interleave_{p-var,[s,t]}+\|y-y^{\delta}\|_{\infty,[s,t]}+\interleave R^{y}-R^{y^{\delta}}\interleave_{q-var,[s,t]^{2}}\right)\nonumber\\
&~~~~~~+C_{p}(C_{g}^{2}\vee C_{g})\interleave \omega^{1}\!-\!W_{\delta}(\cdot,\omega^{1})\interleave_{p\!-\!var,[s,t]}\nonumber\\
&~~~~~~\times\left(\interleave y^{\delta}\interleave_{p-var,[s,t]}\interleave W_{\delta}(\cdot,\omega^{1})\interleave_{p-var,[s,t]}\!+\!\interleave R^{y^{\delta}}\interleave_{q\!-\!var,[s,t]^{2}}\right)\nonumber\\
&~~~~~~+2C_{g}^{2}C_{p}\interleave y^{\delta} \interleave_{p-var,[s,t]}\interleave\omega^{2}-\mathbb{W}_{\delta}(\omega^{1})\interleave_{q-var,[s,t]^{2}}.
\end{align}
\end{small}
For $B_{2}$, we have
\begin{align}\label{5.37}
B_{2}&=\|Dg(y_{s})g(y_{s})\omega^{2}_{s,t}-Dg(y^{\delta}_{s})g(y^{\delta}_{s})\mathbb{W}_{\delta}(\omega^{1})_{s,t}\|\nonumber\\
&\leq \|Dg(y_{s})g(y_{s})\!\!-\!\!Dg(y_{s}^{\delta})g(y_{s}^{\delta})\|\|\omega^{2}_{s,t}\|\!\!+\!\!\|Dg(y_{s}^{\delta})g(y_{s}^{\delta})\|\|\omega^{2}_{s,t}\!\!-\!\!\mathbb{W}_{\delta}(\omega^{1})_{s,t}\|\nonumber\\
&\leq 2C_{g}^{2}\|y\!\!-\!\!y^{\delta}\|_{\infty,[s,t]}\interleave\omega^{2}\interleave_{q-var,[s,t]^{2}}\!+\!C_{g}^{2}\interleave\omega^{2}\!-\!\mathbb{W}_{\delta}(\omega^{1}) \interleave_{q-var,[s,t]^{2}}.
\end{align}
Hence, we can get for $A_{2}$
\begin{small}
\begin{align}\label{5.38}
&A_{2}\leq 14C_{p}(C_{g}^{2}\interleave \boldsymbol{\omega}\interleave^{2}_{p-var,[s,t]}\vee C_{g}\interleave \boldsymbol{\omega}\interleave_{p-var,[s,t]})\nonumber\\
&~~~~~~\times\left(\interleave y^{\delta}\interleave_{p-var,[s,t]}+\interleave y\interleave_{p-var,[s,t]}+\interleave R^{y}\interleave_{q-var,[s,t]^{2}}+1\right)\nonumber\\
&~~~~~~\times\left(\interleave y-y^{\delta}\interleave_{p-var,[s,t]}+\|y-y^{\delta}\|_{\infty,[s,t]}+\interleave R^{y}-R^{y^{\delta}}\interleave_{q-var,[s,t]^{2}}\right)\nonumber\\
&~~~~~~+C_{p}(C_{g}^{2}\vee C_{g})\interleave \omega^{1}\!-\!W_{\delta}(\cdot,\omega^{1})\interleave_{p\!-\!var,[s,t]}\biggr(\interleave R^{y^{\delta}}\interleave_{q-var,[s,t]^{2}}\nonumber\\
&~~~~~~+\interleave y^{\delta}\interleave_{p\!-\!var,[s,t]}(\interleave W_{\delta}(\cdot,\omega^{1})\interleave_{p\!-\!var,[s,t]}\!+\!\interleave\omega^{1}\interleave_{p\!-\!var,[s,t]})\biggr)\nonumber\\
&~~~~~~+2C_g^2C_p[\interleave y^\delta\interleave_{p-var,[s,t]}+1]\interleave \omega^2-\mathbb{W}_\delta(\omega^1)\interleave_{q-var,[s,t]^2}.
\end{align}
\end{small}
By \eqref{5.19},\eqref{5.20}, \eqref{5.38}, we complete the proof.
\end{proof}
\begin{theorem}\label{theorem 7.2}
 Let $r\in(0,1)$, and $\tau\in \mathbb{R}, T>0$. For any sequence of  stopping times on interval $[\tau,\tau+T]$

 $$\tau_{0}=\tau,\quad\tau_{i+1}=\inf\{t>\tau_{i};\interleave\boldsymbol{\omega}\interleave_{p-var,[\tau_{i},t]}=r\}\wedge (\tau+T)$$
 and

$$\tau^{\delta}_{0}=\tau,\quad\tau^{\delta}_{i+1}=\inf\{t>\tau^{\delta}_{i};\interleave\boldsymbol{\omega}^{\delta}\interleave_{p-var,[\tau^{\delta}_{i},t]}=r\}\wedge (\tau+T),$$
then we have

$$\tau^{\delta}_{i}\rightarrow \tau_{i}$$
for any $i\in N_{\frac{1}{4C_{p}C_{g}},[\tau,\tau+T],p}(\boldsymbol{\omega})$.
 \end{theorem}
\begin{proof}
Suppose  $\tau^{\delta}_{i}\nrightarrow \tau_{i}$, in view of  $\tau^{\delta}_{i}\in [\tau,\tau+T]$, then  there exists a sequence $\{\delta_{n}\}_{n\in\mathbb{N}}$, such that $\lim_{n\rightarrow \infty}\tau_{i}^{\delta_{n}}=\tilde{\tau}_{i}\neq \tau_{i}.$  Theorem \ref{theorem 4.6} shows that
 $$\lim_{\delta\rightarrow 0}\interleave \boldsymbol{\omega}_{\delta}\interleave_{p-var,[\tau_{i-1},\tilde{\tau}_{i}]}=\interleave \boldsymbol{\omega}\interleave_{p-var,[\tau_{i-1},\tilde{\tau}_{i}]},$$
 then for any $\epsilon>0$, there exists  a constant $N_{1}(\tilde{\tau}_{i},\epsilon)>0$, such that $n>N_{1}(\tilde{\tau}_{i},\epsilon)$, we have
$$\left|\interleave \boldsymbol{\omega}_{\delta_n}\interleave_{p-var,[\tau_{i-1},\tilde{\tau}_{i}]}-\interleave \boldsymbol{\omega}\interleave_{p-var,[\tau_{i-1},\tilde{\tau}_{i}]}\right|\leq \frac{\epsilon}{2}.
$$
Furthermore, the $\interleave \boldsymbol{\omega}_{\delta}\interleave_{p-var,[\tau_{i-1},t]}$ as the function of $t$, it is continuous for variable $t$. Indeed, 
%
 for fixed $t_{0}>0$, define the following truncated function
$$
W_{\delta}^{t_{0}}\left(s, \omega^{1}\right)=\left\{\begin{array}{ll}
W_{\delta}\left(s, \omega^{1}\right), & s\leq t_{0};\\
W_{\delta}\left(t_{0}, \omega^{1}\right), & s > t_{0}.
\end{array}\right.
$$
 $\mathbb{W}_{\delta}^{t_{0}}$ is the second order process of a smooth rough path $\boldsymbol{\omega}_{\delta}^{t_{0}} =(W_{\delta}^{t_{0}}\left(\cdot, \omega^{1}\right),\mathbb{W}_{\delta}^{t_{0}})$ 
namely, according to the definition of the smooth second order process, we have that
 $$
\mathbb{W}^{t_{0}}_{\delta,s,u}\left( \omega^{1}\right)=\left\{\begin{array}{ll}
\mathbb{W}_{\delta,s,u}\left( \omega^{1}\right), & s<u\leq t_{0};\\
\mathbb{W}_{\delta,s,t_{0}}\left( \omega^{1}\right), & s \leq t_{0}<u;\\
0,& t_{0}\leq s<u.
\end{array}\right.
$$
Thus, for $t\geq t_{0}$, using Lemma \ref{Lemma 2.1} for the second order  process we have
\begin{align}
&\left|\interleave \boldsymbol{\omega}_{\delta}\interleave_{p-var,[0,t]}-\interleave \boldsymbol{\omega}_{\delta}\interleave_{p-var,[0,t_{0}]}\right|\nonumber\\
&~~= \left|\interleave \boldsymbol{\omega}^{\delta}\interleave_{p-var,[0,t]}- \interleave \boldsymbol{\omega}_{\delta}^{t_{0}}\interleave_{p-var,[0,t]}\right|\leq \interleave \boldsymbol{\omega}_{\delta}-\boldsymbol{\omega}^{t_{0}}_{\delta}\interleave_{p-var,[0,t]}\nonumber\\
&~~=\left(\interleave W_{\delta}(\cdot,\omega^{1})\interleave^p_{p-var,[t_{0},t]}+2^{q-1} \interleave \mathbb{W}_{\delta}(\omega^{1})\interleave^{q}_{q-var,[t_{0},t]^{2}}\right)^{\frac{1}{p}}\nonumber\nonumber\\
&~~\leq \interleave W_{\delta}(\cdot,\omega^{1})\interleave_{p-var,[t_{0},t]}+2^{\frac{q-1}{p}} \interleave \mathbb{W}_{\delta}(\omega^{1})\interleave^{\frac{1}{2}}_{q-var,[t_{0},t]^{2}}\nonumber\\
&~~\leq  C(p,q)\interleave \boldsymbol{\omega}_{\delta}\interleave_{\beta,[t_0,t]}(t-t_0)^{\beta},
\end{align}
where we use $\mathcal{C}^{\beta}(I;R^{d})\subset \mathcal{C}^{p-var}(I;R^{d})$ to guarantee the last inequality holds.  Similarly, we can get
$$\left|\interleave \boldsymbol{\omega}_{\delta}\interleave_{p-var,[0,t]}-\interleave \boldsymbol{\omega}_{\delta}\interleave_{p-var,[0,t_{0}]}\right|\leq C(p,q)\interleave \boldsymbol{\omega}_{\delta}\interleave_{\beta,[t,t_0]}(t_0-t)^{\beta}.$$
In view of Theorem  \ref{theorem 4.6}, we know that  $\interleave \boldsymbol{\omega}_{\delta}\interleave_{\beta,[t,t_0]}$ and  $\interleave \boldsymbol{\omega}_{\delta}\interleave_{\beta,[t_0,t]}$ are uniform bounded with respect to $\delta$, then the continuity property is true.  Hence, there exists a constant $C>0$ such that
$$\left|\interleave \boldsymbol{\omega}_{\delta_{n}}\interleave_{p-var,[\tau_{i-1},\tau^{\delta_{n}}_{i}]}-\interleave \boldsymbol{\omega}_{\delta_{n}}\interleave_{p-var,[\tau_{i-1},\tilde{\tau}_{i}]}\right|\leq C|\tau^{\delta_{n}}_{i}-\tilde{\tau}_{i}|^\beta.
$$
Then for any $\epsilon>0$, there exists a constant $N_{2}(\tilde{\tau}_{i},\epsilon)$ such that for $n>N_{2}$, we have
$$|\tau^{\delta_{n}}_{i}-\tilde{\tau}_{i}|<\left(\frac{\epsilon}{2C}\right)^{\frac{1}{\beta}}.$$
Let $N=\max\{N_{1}(\tilde{\tau}_{i},\epsilon),N_{2}(\tilde{\tau}_{i},\epsilon)\}$, then for $n>N$, we obtain
\begin{small}
\begin{align}
&\left|\interleave \boldsymbol{\omega}_{\delta_{n}}\interleave_{p-var,[\tau_{i-1},\tau^{\delta_{n}}_{i}]}-\interleave \boldsymbol{\omega}\interleave_{p-var,[\tau_{i-1},\tilde{\tau}_{i}]}\right|\nonumber\\
&\leq \left|\interleave \boldsymbol{\omega}_{\delta_{n}}\interleave_{p-var,[\tau_{i-1},\tau^{\delta_{n}}_{i}]}-\interleave \boldsymbol{\omega}_{\delta_{n}}\interleave_{p-var,[\tau_{i-1},\tilde{\tau}_{i}]}\right|\nonumber\\&~~+\left|\interleave \boldsymbol{\omega}_{\delta_{n}}\interleave_{p-var,[\tau_{i-1},\tilde{\tau}_{i}]}-\interleave \boldsymbol{\omega}\interleave_{p-var,[\tau_{i-1},\tilde{\tau}_{i}]}\right|\nonumber\\
&\leq \epsilon.
\end{align}
\end{small}
Thus, we get
$$r=\lim_{n\rightarrow \infty}\interleave \boldsymbol{\omega}_{\delta_{n}}\interleave_{p-var,[\tau_{i-1},\tau^{\delta_{n}}_{i}]}=\interleave \boldsymbol{\omega}\interleave_{p-var,[\tau_{i-1},\tilde{\tau}_{i}]}.$$
By the variation norm as a function with variable $t$ is strictly increasing and continuous, and $\tilde{\tau}_{i}\neq \tau_{i}$, this is contradiction. Then for all sequences $\{\delta_{n}\}$, we have
$$\lim_{n\rightarrow \infty}\tau_{i}^{\delta_{n}}=\tau_{i},$$
namely,
$$\lim_{\delta\rightarrow 0}\tau_{i}^{\delta}=\tau_{i}.$$
\end{proof}
\small
\bibliographystyle{abbrv}

\end{document}